\documentclass[12pt,letterpaper]{amsart}

\usepackage{color}
\usepackage{amssymb, bbm}
\usepackage{graphicx}

\newcommand{\D}{{\partial}}
\newcommand \ow {\overline{w}}
\newcommand \ou {\overline{u_1}}

\newcommand{\RR}{{\mathbb R}}

\newcommand{\defeq}{\stackrel{\rm{def}}{=}}

\newcommand{\NN}{{\mathbb N}}

\renewcommand{\l}{\lambda}
\newcommand{\eps}{\varepsilon}
\newcommand{\supp}{\operatorname{supp}}

\newcommand{\tr}{\operatorname{tr}}

\renewcommand{\Re}{\mathop{\rm Re}\nolimits}

\theoremstyle{plain}

\newtheorem{thm}{Theorem}
\newtheorem{prop}{Proposition}[section]

\newtheorem{lem}[prop]{Lemma}

\theoremstyle{definition}

\newtheorem{rem}{Remark}[section]

\numberwithin{equation}{section}

\def\squarebox#1{\hbox to #1{\hfill\vbox to #1{\vfill}}} 
 
\newcommand{\sech}{\textnormal{sech}}

\newcommand{\indentalign}{\hspace{0.3in}&\hspace{-0.3in}}
\newcommand{\la}{\langle}
\newcommand{\ra}{\rangle}

\newcommand{\nlso}{\textnormal{NLS}_0}
\newcommand{\nlsq}{\textnormal{NLS}_q}
\usepackage{amsxtra}
\newcommand{\trans}{\textnormal{tr}}
\newcommand{\refl}{\textnormal{ref}}
\newcommand{\bd}{\textnormal{bd}}

\ifx\pdfoutput\undefined
  \DeclareGraphicsExtensions{.pstex, .eps}
\else
  \ifx\pdfoutput\relax
    \DeclareGraphicsExtensions{.pstex, .eps}
  \else
    \ifnum\pdfoutput>0
      \DeclareGraphicsExtensions{.pdf}
    \else
      \DeclareGraphicsExtensions{.pstex, .eps}
    \fi
  \fi
\fi

\title
[Fast soliton scattering by attractive delta impurities]
{Fast soliton scattering by attractive delta impurities}

\author[Kiril Datchev]
{Kiril Datchev}
\email{datchev@math.berkeley.edu}
\author[Justin Holmer]
{Justin Holmer}
\email{holmer@math.berkeley.edu}
\address{Mathematics Department, University of California \\
Evans Hall, Berkeley, CA 94720, USA}

\begin{document}    

\begin{abstract}
We study  the Gross-Pitaevskii equation with an attractive delta function
potential and show that in the high velocity limit an incident soliton is
split into reflected and transmitted soliton components plus a small
amount of dispersion.  We give explicit analytic formulas for the reflected
and transmitted portions, while the remainder takes the form of an error.
Although the existence of a bound state for this potential introduces
difficulties not present in the case of a repulsive potential, we
show that the proportion of the soliton which is trapped at the origin vanishes in the limit.
\end{abstract}
   
\maketitle

\section{Introduction}   
\label{in}

The nonlinear Schr\"odinger equation (NLS) or Gross-Pitaevskii equation (GP)
\begin{equation}
\label{E:fNLS}
i\partial_t u + \tfrac12\partial_x^2 u + |u|^2u=0 \, ,
\end{equation}
where $u=u(x,t)$ and $x\in \mathbb{R}$, possesses a family of soliton solutions
$$u(x,t) = e^{i\gamma} e^{ivx}e^{-\frac12itv^2}\lambda \sech(\lambda(x-x_0-vt))$$
parameterized by the constants of phase $\gamma\in \mathbb{R}$, velocity $v\in \mathbb{R}$, initial position $x_0\in \mathbb{R}$, and scale $\lambda>0$.  Given that these solutions are exponentially localized, they very nearly solve the perturbed equation
\begin{equation}
\label{E:pNLS}
i\partial_t u + \tfrac12\partial_x^2 u - q\delta_0(x) + |u|^2u=0\, ,
\end{equation}
when the center of the soliton $|x_0+vt| \gg 1$.  In fact, if we consider initial data
$$u(x,0) = e^{ixv}\sech (x-x_0)$$
for $x_0 \ll -1$, then we expect the solution to essentially remain the rightward propagating soliton $e^{ixv}e^{-\frac12 iv^2t} \sech(x-x_0-vt)$ until time $t\sim |x_0|/v$ at which point a substantial amount of mass ``sees'' the delta potential.  It is of interest to examine the subsequent behavior of the solution, as this arises as a model problem in nonlinear optics and condensed matter physics (see Cao-Malomed \cite{CM} and Goodman-Holmes-Weinstein \cite{GHW}). In the case  $|q| \ll 1$, Holmer-Zworski \cite{HZ1} find that the soliton remains intact and the evolution of the center of the soliton approximately obeys Hamilton's equations of motion for a suitable effective Hamiltonian.    This result applies to both the repulsive ($q>0$) case and the attractive ($q<0$) case, and identifies the $|q|\ll 1$ setting as a \emph{semi-classical} regime.


On the other hand, \emph{quantum} effects dominate for high velocities $|v|\gg 1$.  In Holmer-Marzuola-Zworski \cite{HMZ}\cite{HMZ2}, the case of $q>0$ and $v \gg 1$ is studied (most interesting is the regime $q\sim v$), and it is proved that the incoming soliton is split into a transmitted component and a reflected component.   The transmitted component continues to propagate to the right at velocity $v$ and the reflected component propagates back to the left at velocity $-v$, see Fig. \ref{F:snapshots}.  The transmitted mass and reflected mass are determined as well as the detailed asymptotic form of the transmitted and reflected waves.  The rigorous analysis in \cite{HMZ} is rooted in the heuristic that at high velocities, the time of interaction of the solution with the delta potential is short, and thus the solution is well-approximated in $L^2$ by the solution to the corresponding linear problem
\begin{equation}
\label{E:plin}
i\partial_t u + \tfrac12\partial_x^2 u -q\delta_0(x) u =0 \,.
\end{equation}
This heuristic is typically valid provided the problem is $L^2$ subcritical with respect to scaling.  In this case, it is shown to hold using Strichartz estimates for solutions to this linear problem and its inhomogeneous counterpart, with bounds independent of $q$.  The Strichartz estimates are also used in a perturbative analysis comparing the incoming solution (pre-interaction) and outgoing solution (post-interaction) with the solution to the free NLS equation \eqref{E:fNLS}.  One then proceeds with an analysis of the linear problem to understand the interaction.  Let $H_q = -\frac12\partial_x^2 + q\delta_0(x)$ and consider a general plane wave solution to $(H_q-\frac12\lambda^2)w=0$,
$$w(x) = 
\left\{
\begin{aligned}
&A_+ e^{-i\lambda x} + B_- e^{i\lambda x} & \text{for } x>0\\
&A_- e^{-i\lambda x} + B_+ e^{i\lambda x} & \text{for } x<0
\end{aligned}
\right.
$$
The matrix
$$S(\lambda): \begin{bmatrix} A_+ \\ B_+ \end{bmatrix} \mapsto \begin{bmatrix} A_- \\ B_- \end{bmatrix}$$
sending incoming ($+$) coefficients to outgoing ($-$) coefficients is called the scattering matrix, and in this case it can be easily computed as
$$S(\lambda) = \begin{bmatrix} t_q(\lambda) & r_q(\lambda) \\ r_q(\lambda) & t_q(\lambda) \end{bmatrix} \, ,$$
where $t_q(\lambda)$ and $r_q(\lambda)$ are the transmission and reflection coefficients
$$t_q(\lambda) = \frac{i\lambda}{i\lambda-q} \quad \text{and} \quad r_q(\lambda) = \frac{q}{i\lambda -q}\, .$$
We have that at high velocities and for $x_1\ll -1$,
\begin{equation}
\label{E:linscat}
e^{-itH_q}[e^{ixv}\sech(x-x_1)] \approx t(v)e^{-itH_0}[e^{ixv}\sech(x-x_1)] + r(v) e^{-itH_0}[e^{-ixv}\sech(x+x_1)].
\end{equation}
From this we can infer that the transmitted mass 
$$T_q(v) = \frac{\|u(t)\|_{L_{x>0}^2}^2}{\|u(t)\|_{L_x^2}^2} = \tfrac12\|u(t)\|_{L_{x>0}^2}^2$$
matches the quantum transmission rate at velocity $v$, i.e. the square of the transmission coefficient
$$T_q(v) \approx |t_q(v)|^2 = \frac{v^2}{v^2+q^2}$$
This is confirmed by a numerical analysis of this problem in Holmer-Marzuola-Zworski \cite{HMZ2}, where it is reported that for $q/v$ fixed,
$$T_q(v) = \frac{v^2}{v^2+q^2} + \mathcal{O}(v^{-2}), \quad \text{as }v\to +\infty \, .$$
Further, \eqref{E:linscat} gives approximately the form of the solution just after the interaction, and one can then model the post-interaction evolution by the free nonlinear equation \eqref{E:fNLS} and apply the inverse scattering method to yield a detailed asymptotic.  The results of \cite{HMZ} are valid up to time $\log v$, at which point the errors accumulated in the perturbative analysis become large.

\begin{figure}
\begin{center}
\includegraphics[scale=0.42]{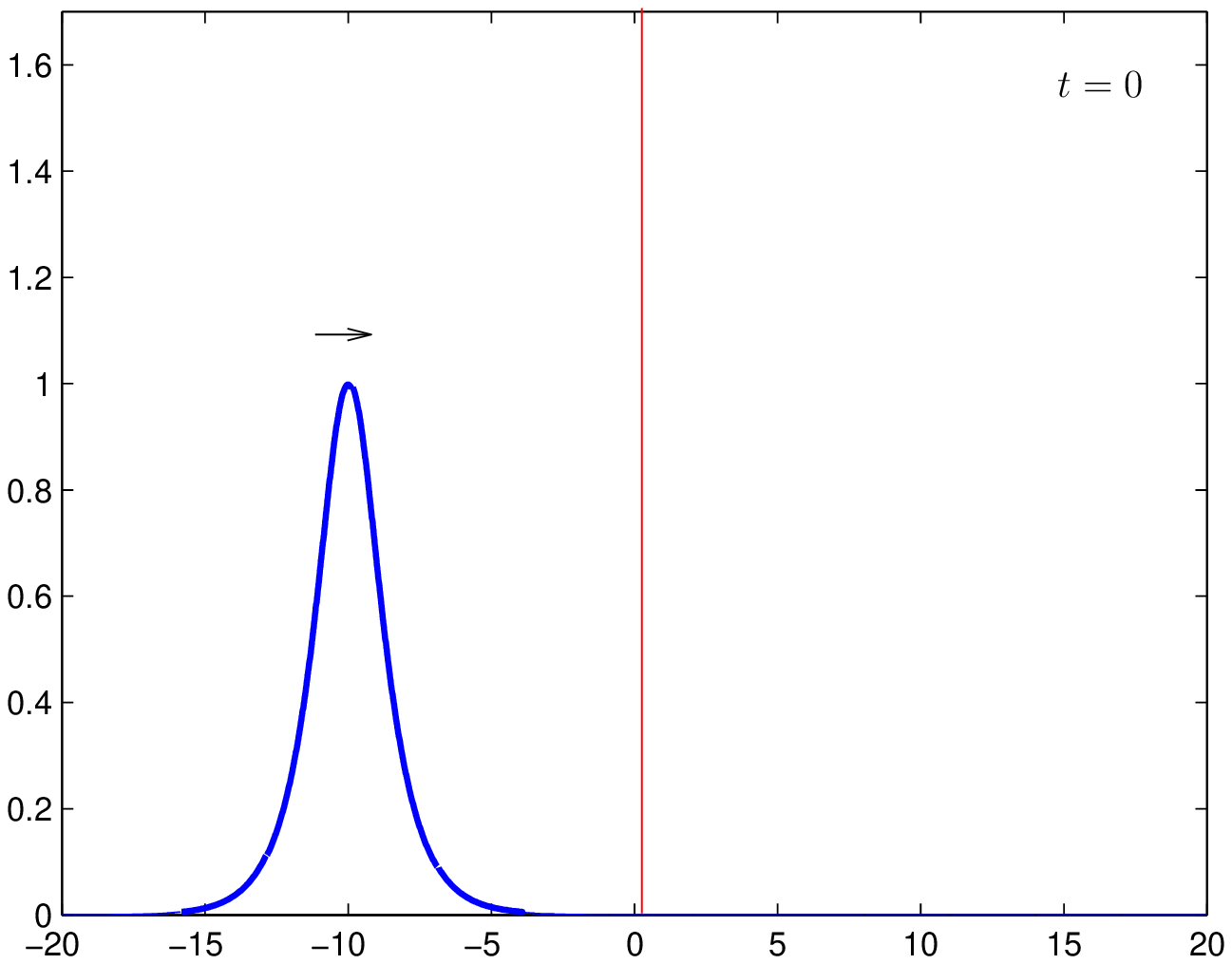}\hfill
\includegraphics[scale=0.42]{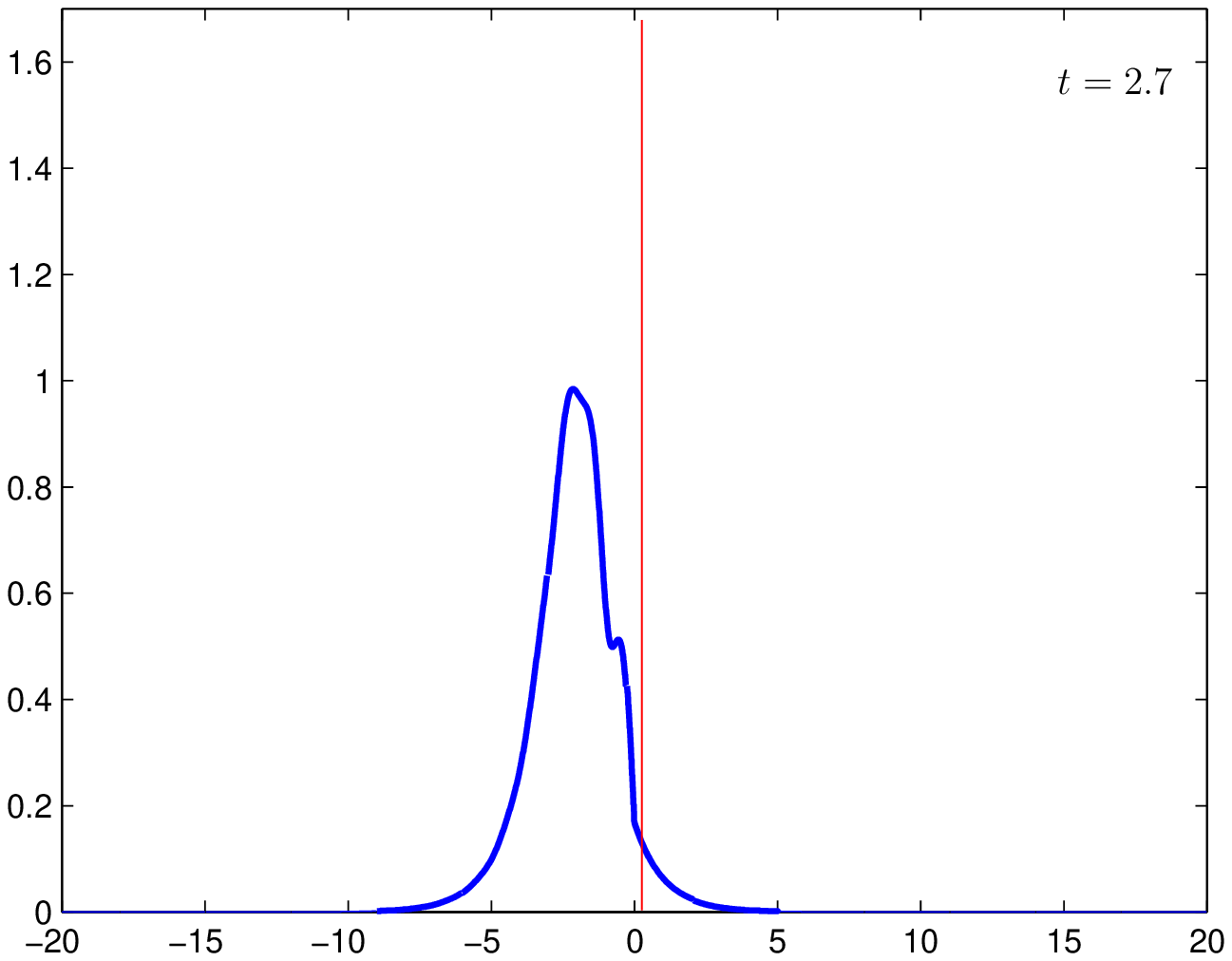}
\includegraphics[scale=0.42]{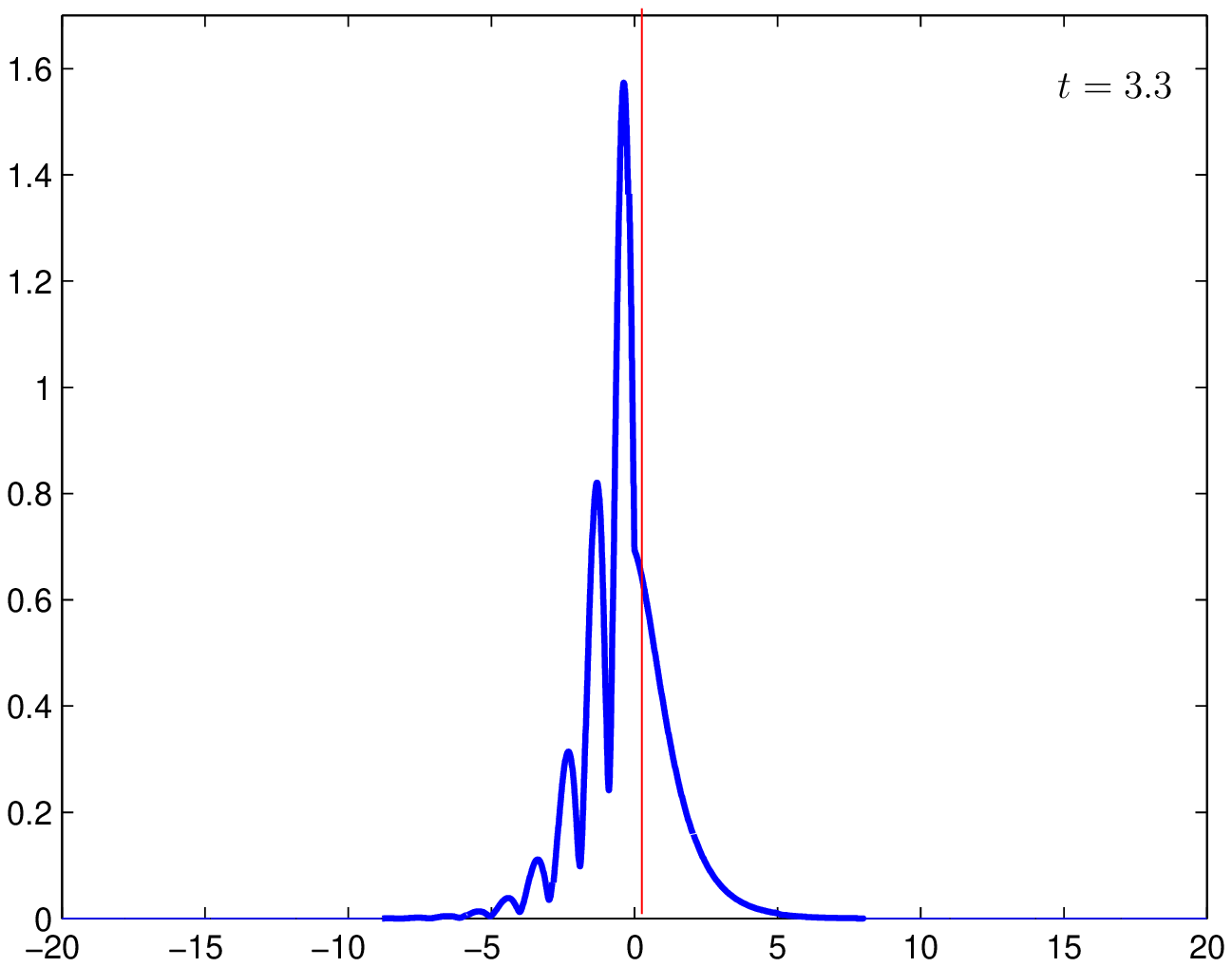}\hfill
\includegraphics[scale=0.42]{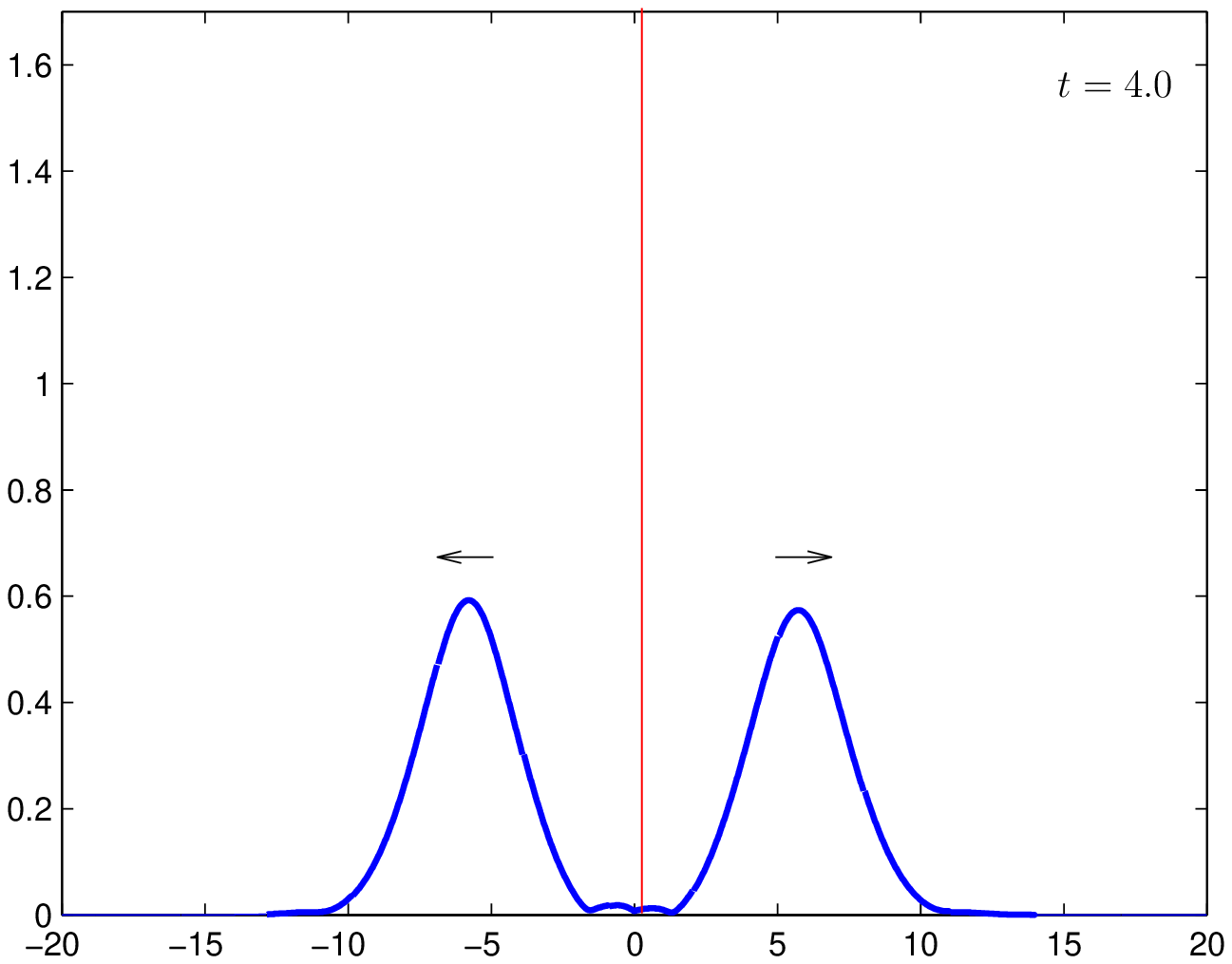}
\end{center}
\caption{\label{F:snapshots} Numerical simulation of the case $q=-3$, $v=3$, $x_0=-10$, at times $t=0.0, 2.7, 3.3, 4.0$.  Each frame is a plot of amplitude $|u|$ versus $x$.}
\end{figure}

When $q<0$, the nonlinear equation \eqref{E:pNLS} has a one-parameter family of bound state solutions 
\begin{equation}
\label{E:bound}
u(x,t) = e^{it\lambda^2/2}\lambda \sech(\lambda|x|+\tanh^{-1}(|q|/\lambda)), \quad 0<|q|<\lambda
\end{equation}
The numerical simulations in \cite{HMZ2} show that at high velocities, the incoming soliton is still split into a rightward propagating transmitted component and a leftward propagating reflected component, although in addition some mass is left behind at the origin ultimately resolving to a bound state of the form \eqref{E:bound}.  However, the amount of mass trapped at the origin diminishes exponentially as $v\to +\infty$ and the observed mass of the transmitted and reflected waves is consistent with the assumption that the outgoing solution is still initially well-modelled by \eqref{E:linscat}.

In this paper, we undertake a rigorous analysis of the $q<0$ and $v$ large case.  This analysis is complicated by the presence an eigenstate solution $u(x,t) = e^{\frac12 itq^2}e^{-|q||x|}$ to the linear problem \eqref{E:plin}.  Therefore, the Strichartz estimates, which involve global time integration, cannot be valid for general solutions to \eqref{E:plin}.  However, they can be shown to hold for the \emph{dispersive} component of the solution $e^{-itH_q}(\phi - P\phi)$, where $P$ is the orthogonal projection onto the eigendirection $e^{-|q||x|}$.   In the pre-interaction, interaction, and post-interaction perturbative analyses, this eigenstate must be separately analyzed.  This introduces the most difficulty in the post-interaction analysis, although (as explained in more detail below), we are able to obtain suitable estimates by introducing a more refined decomposition of the outgoing waves and invoking some nonlinear energy estimates.  We thus obtain the following:

\begin{thm}
\label{th:1}
Fix  $0 < \eps \ll 1$.  If $u(x,t)$ is the solution of \eqref{E:pNLS} with initial condition $u(x,0)=e^{ixv}\sech(x-x_0)$ and $x_0\leq -v^\eps$, then for $|q|\gtrsim 1$ and 
\begin{equation}
\label{E:vlarge}
v \ge C(\log |q|)^{1/\eps} + C|q|^{\frac{13}{14}(1+2\eps)} + C_{\eps,n} \la q \ra^{\frac 1n}
\end{equation}
we have
\begin{equation}
\label{eq:th}
\frac{1}{2} \int_{x> 0 } | u ( x , t ) |^2 dx = \frac{v^2}{ v^2 + q^2 } + 
{\mathcal O}( |q|^\frac13v^{-\frac76(1-2\eps)})+\mathcal{O}(v^{-(1-2\eps)})
\end{equation} 
uniformly for ``post-interaction'' times 
$$\frac{|x_0|}{v}+v^{-1+\eps}\leq t \leq \eps\log v. $$  
Here the constant $C$ and the constants in $\mathcal{O}$ are independent of $q,v,\eps$, while $C_{\eps,n}$ is a constant depending on $\eps$ and $n$ which goes to infinity as $\eps \to 0$ or $n \to \infty$.
\end{thm}

The proof is outlined in \S\ref{proof}.  It is decomposed into estimates for the pre-interaction phase (Phase 1), interaction phase (Phase 2), and post-interaction phase (Phase 3).  The details of the estimates for each of the phases are then given in \S\ref{S:phase1} (Phase 1), \S\ref{S:phase2} (Phase 2), and \S\ref{S:phase3} (Phase 3).

The assumption that $v \gtrsim |q|^{\frac{13}{14}(1+2\eps)}$ is new to our $q<0$ analysis; no assumption of this strength was required in the $q>0$ case treated in \cite{HMZ}.  It is needed in order to iterate over unit-sized time intervals in the post-interaction phase.  The perturbative equation in that analysis has a forcing term whose size can be at most comparable to the size of the initial error.  The condition that emerges is $|q|^{3/2}(\text{error})^2 \leq c$.  Since the error bestowed upon us from the interaction phase analysis is $|q|^\frac13v^{-\frac76(1-\eps)}$, the condition $|q|^{3/2}(\text{error})^2 \leq c$ equates to $v \gtrsim |q|^{\frac{13}{14}(1+2\eps)}$.  Provided this condition is satisfied, we can interate over $\sim \eps \log v$ unit-sized time intervals, with the error bound doubling over each interval, and incur a loss of size $v^{\eps}$.  This enables us to reach time $\eps \log v$.

The assumption $v \gtrsim |q|^{\frac{13}{14}(1+2\eps)}$ is not a serious limitation, however, since the most interesting phenomenom (even splitting or near even splitting) occurs for $|q|\sim v$.  Furthermore, if the analysis is only carried through the interaction phase (ending at time $|x_0|/v+v^{-1+\eps}$) and no further, then only the assumption $v \gtrsim |q|^{\frac12(1+\eps)}$ is needed.  We believe that if our post-interaction arguments are amplified with a series of technical refinements, we could relax the restriction needed there from $v \gtrsim |q|^{\frac{13}{14}(1+2\eps)}$ to $v \gtrsim |q|^{\frac12(1+\eps)}$.  On the other hand, the condition $v \gtrsim |q|^{\frac12(1+\eps)}$ shows up in a more serious way in the interaction analysis, and to relax this restriction even further (if it is possible) would require a more significant new idea.

The condition $v \gtrsim |q|^{\frac12(1+\eps)}$ comes about as a result of applying Strichartz estimates to the flow $e^{-itH_q}\phi$ rather than just to the dispersive part $e^{-itH_q}(\phi-P\phi)$, and the additional error found in the $q<0$ case compared to the $q>0$ case arises in the same way.  As discussed in Theorem \ref{th:3} below, in the case $q>0$ we have $\mathcal{O}(v^{-1+})$ in place of ${\mathcal O}( |q|^\frac13v^{-\frac76+})+\mathcal{O}(v^{-1+})$ for $q<0$ (which in the crucial regime $|q| \sim v$ becomes $\mathcal{O}(v^{-\frac 56 +})$).    However, the numerical study conducted in \cite{HMZ2} (see equation (2.4), Table 2, and Fig.\ 5 in that paper) suggests that the trapping at the origin should be exponentially small instead, indicating that this is probably only an artifact of our method of proof.
 
The proof of Theorem \ref{th:1} is based entirely upon estimates for the perturbed and free linear propagators, and some nonlinear conservation laws (energy and mass); there is no use of the inverse scattering theory.  However, as in \cite{HMZ}, we can combine the inverse scattering theory with the proof of Theorem \ref{th:1} to obtain a strengthened result giving more information about the behavior of the outgoing waves.  This result we state as:

\begin{thm}
\label{th:2} 
Under the hypothesis of Theorem \ref{th:1} and for 
$$ \frac{|x_0|}{v}+1 \leq  t  \leq \eps \log v, $$ 
we have 
\begin{equation}
\label{eq:th2}
u(x,t) = 
\begin{aligned}[t]
&\phi_0(|t_q(v)|) e^{\frac{1}{2}i|\tilde T_q(v)|^2t}e^{i\arg t_q(v)} e^{ixv}e^{-itv^2}\tilde T_q(v) \sech(\tilde T_q(v)(x-x_0-tv))\\
&+\phi_0(|r_q(v)|) e^{\frac{1}{2}i|\tilde R_q(v)|^2t}e^{i\arg r_q(v)} e^{-ixv}e^{-itv^2}\tilde R_q(v) \sech(\tilde R_q(v)(x+x_0+tv))\\
& + \mathcal{O}_{L_x^\infty}\left(\left(t-\frac{|x_0|}{v}\right)^{-1/2}\right) +\mathcal{O}_{L_x^2}(|q|^{\frac13}v^{-\frac76(1-2\eps)})+\mathcal{O}(v^{-1+2\eps})
\end{aligned}
\end{equation}
where
\begin{equation}
\label{eq:th3}
\tilde T_q(v)=[2|t_q(v)|-1]_+, \qquad \tilde R_q(v)=[2|r_q(v)|-1]_+\, ,
\end{equation}
$$\phi_0(\alpha) = \int_0^\infty \log\left( 1 + \frac{\sin^2\pi \alpha}{\cosh^2\pi \zeta} \right) \frac{\zeta}{\zeta^2+(2\alpha-1)^2} \, d\zeta$$
When $ 2 | t_q ( v ) | =1 $ or $ 2 | r_q ( v ) | = 1 $ the first error term in \eqref{eq:th2} is modified to $ {\mathcal O}_{ L^\infty_x } ( (\log  ( t - |x_0|/v ))/( t - |x_0|/v ))^{\frac12} ) $.  
\end{thm}

The proof of Theorem \ref{th:2} is not discussed in the main body of this paper, since all of the needed information is contained in \S4 and Appendix B of \cite{HMZ}.    The main point is that Theorem \ref{th:1} in fact establishes that for times $|x_0|/v+1\leq t\leq \eps\log v$, we have
$$ u(x,t) = 
\begin{aligned}[t]
& e^{-itv^2/2}e^{it_2/2}e^{ixv}\nlso(t-t_2)[t(v)\sech(x)](x-x_0-tv) \\
&+ e^{-itv^2/2}e^{it_2/2}e^{-ixv}\nlso(t-t_2)[r(v)\sech(x)](x+x_0+tv) \\
&+\mathcal{O}(v^{-(1-\eps)}) + \mathcal{O}(|q|^\frac 13v^{-\frac 76(1-2\eps)})
\end{aligned}
$$
where $\nlso(t)\phi$ denotes the free nonlinear flow according to \eqref{E:fNLS}.  This is the starting point of the arguments provided in \S4 and Appendix B of \cite{HMZ}, which carry out an asymptotic (in time) description of the free nonlinear evolution of $\alpha\sech x$, for a constant $0\leq \alpha<1$.

Although the main point of the present paper is to handle the difficulties involved in the case $q<0$ stemming from the presence of a linear eigenstate, some of the refinements we introduce (specifically, cubic correction terms in the interaction phase analysis) improve the result of \cite{HMZ} in the case $q>0$.  In fact, these refinements are simpler when carried out for $q>0$ directly, and we therefore write them out separately in that setting in \S \ref{S:phase2pos}.  We summarize the results as:

\begin{thm}
\label{th:3}
In the case $q>0$, the assumption \eqref{E:vlarge} in Theorem \ref{th:1} can be replaced by the less restrictive $ v \ge C(\log |q|)^{1/\eps} + C_{\eps,n} \la q \ra^{\frac 1n}$, and the conclusion \eqref{eq:th} holds with the first error term dropped (that is, $\mathcal{O}_{L_x^2}(|q|^{\frac13}v^{-\frac76(1-2\eps)})$ is dropped and only $\mathcal{O}(v^{-1+2\eps})$ is kept).  Also, the conclusion of Theorem 2 holds with $\mathcal{O}_{L_x^2}(|q|^{\frac13}v^{-\frac76(1-2\eps)})$ dropped from \eqref{eq:th3}.
\end{thm}

Thus in the $q > 0$ case we improve the $L^2$ error from $\mathcal{O}(v^{-1/2 +})$ to $\mathcal{O}(v^{-1 +})$. It may be possible to to improve this error further to $\mathcal{O}(v^{-2 +})$ using an iterated integral expansion of the error in the spirit of Sections \ref{S:phase2} and \ref{S:phase2pos}, although a more detailed analysis than the one given there would be needed.

We now outline the proof of Theorem \ref{th:1}, the main result of the paper, highlighting the modifications of the argument in \cite{HMZ} needed to address the case of $q<0$.  We will use the following terminology:  the \emph{free linear} evolution is according to the equation $i\partial_t u + \tfrac12\partial_x^2 u =0$, the \emph{perturbed linear} evolution is according to the equation $i\partial_t u + \tfrac12\partial_x^2 u -q\delta_0 u=0$, the \emph{free nonlinear} evolution is according to the equation $i\partial_t u + \tfrac12\partial_x^2 u +|u|^2u=0$, and the \emph{perturbed nonlinear} evolution is according to the equation $i\partial_t u + \tfrac12\partial_x^2 u -q\delta_0 u+|u|^2u=0$.

The analysis breaks into three separate time intervals: Phase 1 (pre-interaction), Phase 2 (interaction), and Phase 3 (post-interaction).  The analysis of Phase 2, discussed in part earlier, is initially based on the principle that at high velocities, the time length of interaction is short $\sim v^{-1+}$, and thus the perturbed nonlinear flow is well-approximated by the perturbed linear flow.  In \cite{HMZ}, this was proved to hold for $q>0$ with a bound on the $L^2$ discrepancy of size $\sim v^{-\frac12+}$.  In the case $q<0$, we suffer some loss in the strength of the estimates due to the flow along the eigenstate $|q|^\frac12e^{-|q||x|}$, and by directly following the approach of \cite{HMZ} the best error bound we could obtain is $\sim |q|^\frac13v^{-\frac23+}+v^{-\frac12+}$.  In the important regime $|q|\sim v$, this gives an error bound of size $v^{-\frac13+}$, which does not suffice to carry through the Phase 3 post-interaction analysis discussed below.    For this reason, we are forced to introduce a cubic correction term to the linear approximation analysis in Phase 2.  The Strichartz based argument then shows that the $L^2$ size of the difference between the solution and the linear flow plus cubic correction is of size $\sim |q|^\frac13v^{-\frac76+} + v^{-1+}$.  However, since the cubic correction term is fairly explicit, we can do a direct analysis of it (not using the Strichartz estimates) and show that it is also of size $v^{-1+}$.  Thus, in the end, we learn that the solution itself is approximated by the perturbed linear flow with error $|q|^\frac13v^{-\frac76+}+v^{-1+}$.  

We then carry out the analysis of the perturbed linear evolution, as disscussed earlier, and show that by the end of the interaction phase, the solution is decomposed into a transmitted component (modulo a phase factor)
\begin{equation}
\label{E:intro_trans}
t(v)e^{ixv}\sech(x-x_0-t_2v)
\end{equation}
and a reflected component (again modulo a phase factor)
\begin{equation}
\label{E:intro_refl}
r(v)e^{-ixv}\sech(x+x_0+t_2v) \,.
\end{equation}
In the post-interaction analysis, we aim to argue that the solution is well-approximated by the free nonlinear flow of \eqref{E:intro_trans} (that we denote $u_{\tr}$) plus the free nonlinear flow of \eqref{E:intro_refl} (that we denote $u_{\refl})$.  It is at this stage that the most serious difficulties beyond those in \cite{HMZ} are encountered.  The approach employed in \cite{HMZ} was to model the solution $u$ as $u= u_{\tr}+u_{\refl}+w$, write the equation for $w$ induced by the equations for $u$, $u_{\tr}$, and $u_{\refl}$, and bound $\|w\|_{L_{[t_a,t_b]}L_x^2}$ over  unit-sized time intervals $[t_a,t_b]$ in terms of the initial size $\|w(t_a)\|_{L^2}$ for that time interval.  This was accomplished by using the Strichartz estimates.  The Strichartz estimates provide a bound on a whole family of space-time norms $\|w\|_{L_{[t_a,t_b]}^qL_x^r}$ where $(q,r)$ are exponents satisfying an admissibilty condition $\frac{2}{q}+\frac1{r}=\frac12$.  This family includes the norm $L_{[t_a,t_b]}^\infty L_x^2$; the other norms (such as $L_{[t_a,t_b]}^6L_x^6$) are needed since they necessarily arise on the right-hand size of the estimates.  From these estimates, we are able to conclude that the error at most doubles over unit-sized time intervals, and thus after $\sim \eps\log v$ time intervals, we have incurred at most an error of size $v^\eps$.

This strategy presents a problem for the case $q<0$, since the linear eigenstate $|q|^{1/2}e^{-|q||x|}$ is well-controlled in $L_{[t_a,t_b]}^\infty L_x^2$ (of size $\sim 1$) but poorly controlled in $L_{[t_a,t_b]}^6L_x^6$ (of size $\sim |q|^\frac13$).  We thus opt to model the post-interaction solution as $u=u_{\tr}+u_{\refl}+u_{\bd}+w$, where $u_{\bd}$ is the \emph{perturbed nonlinear} evolution of the $L_x^2$ projection of $(u(t_a)-u_{\refl}(t_a)-u_{\tr}(t_a))$ onto the linear bound state $|q|^{\frac12}e^{-|q||x|}$.  Then we can use nonlinear estimates based on mass conservation and energy conservation to control the growth of $u_{\bd}$ over the interval $[t_a,t_b]$.  Then $w(t_a)$ is orthogonal to the linear eigenstate, and we can use the Strichartz estimates to control it over the interval $[t_a,t_b]$.  In the estimates, we take care to only evaluate $u_{\bd}$ in one of the norms controlled by mass or energy conservation.   This argument is carried out in detail in \S\ref{S:phase3}.

\noindent
{\sc Acknowledgments.}  We would like to thank  Maciej Zworski for helpful discussions during the preparation of this paper.  The first author was supported in part by NSF grant DMS-0654436 and the second author was supported in part by an NSF postdoctoral fellowship.

\section{Scattering by a delta function}
\label{ros}

Here we present some basic facts about scattering by a $ \delta $-function potential on the real line.  Let $q < 0$ and put
\[ H_q = - \frac{1}2 \frac{d^2}{dx^2} + q \delta_0 ( x ), \qquad H_0 = -\frac 1 2 \frac{d^2}{dx^2} \, .\]
The operator $H_q$ is self-adjoint on the following domain:
\[ \mathcal{D}(H_q) = \{f \in H^2(\RR \setminus \{0\}): f'(0^+) - f'(0^-) = 2q f(0)\},\]
where $f(0)$ means $\lim_{x \to 0} f(x)$ and $f'(0^\pm)$ means $\lim_{x \to 0^\pm} f'(x)$. This can be seen by verifying that the operators $H_q \pm i$ are both symmetric and surjective on $\mathcal{D}(H_q)$. We define special solutions, $ e_\pm ( x , \lambda ) $, to $ ( H_q - \lambda^2 /2 ) e_\pm  = 0 $, as follows
\begin{equation}
\label{211}
e_{\pm}(x,\l) = t_q (\l)e^{\pm i \l x} x_{\pm}^0 
+ (e^{\pm i \l x} + r_q (\l)e^{\mp i\l x}) x_{\mp}^0  \,, 
\end{equation}
where $ t_q $ and $ r_q $ are the transmission and reflection coefficients:
\begin{equation}
\label{eq:tr}
t_q ( \lambda ) = \frac{ i \lambda } { i \lambda - q } \,, \ \ 
r_q ( \lambda ) = \frac{ q} {i \lambda - q } \,.
\end{equation}
They satisfy two equations, one standard (unitarity) and one due to the special structure of the potential:
\begin{equation}
\label{eq:trpr} | t_q ( \lambda ) |^2 + | r_q ( \l ) |^2 = 1 \,, \ \ 
t_q ( \lambda ) = 1 + r_q ( \lambda ) \,.
\end{equation}
Let $P$ denote the $L^2$-projection onto the eigenstate $e^{q|x|}$.  Specifically,  
\begin{equation}
\label{E:Pdef}
P\phi(x) = |q|^{1/2}e^{q|x|} \int_y |q|^{1/2}e^{q|y|}\phi(y)\, dy
\end{equation}
We have $e^{-itH_q}P\phi(x) = e^{\frac{1}{2}itq^2}P\phi(x)$.  Note that $P \phi$ is defined for $\phi \in L_x^r$, $1\leq r \leq \infty$, and by the H\"older inequality,
\begin{equation}
\label{E:PHolder}
\|P\phi\|_{L^{r_2}} \leq c|q|^{\frac{1}{r_1}-\frac{1}{r_2}}\|\phi\|_{L^{r_1}}, \quad 1\leq r_1, r_2 \leq \infty.
\end{equation}

We use the representation of the propagator in terms of the generalized eigenfunctions -- see the notes \cite{TZ} covering scattering by compactly supported potentials.   The resolvent 
\[ R_q ( \lambda ) \defeq ( H_q - \lambda^2 / 2 )^{-1} \,,\]
is given by 
\[R_q ( \l)(x,y) =
\begin{aligned}[t]
\frac{1}{i\l t_q (\l)}\, \big(e_+(x,\l)e_-(y,\l)(x-y)^0_+ + e_+(y,\l)e_-(x,\l)(x-y)^0_-\big)
\end{aligned} \]
Using Stone's thoerem, this gives an explicit formula for the spectral projection, and hence the Schwartz kernel of the propagator:
\begin{equation}
\label{eq:sppr}
e^{ - i t H_q }  =  \frac{1}{2\pi}\int^\infty_0 e^{- i t \lambda^2/2 } \left(e_+(x,\l)\overline{e_+(y,\l)} + e_-(x,\l) \overline{e_-(y,\l)}\right) \,d\l + e^{\frac{1}{2}itq^2}P
\end{equation}
We introduce the following notation for the dispersive part of $e^{ - i t H_q }$:
\[
U_q(t) \defeq e^{ - i t H_q } - e^{\frac{1}{2}itq^2}P.
\]
The propagator for $ H_q  $ is then described in the following
\begin{lem}
\label{p:lin}
Suppose that $ \phi \in L^1 $ and that $ \supp \phi \subset ( -\infty , 0] $.
Then 
\begin{equation}
\label{eq:prop}
e^{-itH_q}\phi(x) = 
\begin{aligned}[t]
& e^{\frac 1 2 itq^2}P\phi(x) + e^{-itH_0}(\phi \ast \tau_q)(x)\, x_+^0 \\
& + (e^{-itH_0}\phi(x) + e^{-itH_0}(\phi\ast \rho_q)(-x))\, x_-^0 
\end{aligned}
\end{equation}
where 
\begin{equation}
\label{eq:ftr}  
\tau_q ( x  ) = \delta_0 ( x) + \rho_q ( x ), \qquad 
\rho_q ( x) = qe^{qx} x_+^0
\end{equation}
\end{lem}

Observe that we have, using a deformation of contour,
\begin{align*}
\hat \rho_q (\l) &= q\int_0^\infty e^{x(q-i\l)} dx = \frac q {q-i\l} \int_0^{-\infty} e^x dx = r_q(\l) \\
\hat \tau_q (\l) &= 1 + r_q(\l) = t_q(\l).
\end{align*}
Observe also that $H_q R = R H_q $, where $R\phi(x) = \phi(-x)$, so that the restriction on the support of $\phi$ is not a serious one, and the formula will allow us to estimate operator norms of $U_q$ using $e^{-itH_0}$. Indeed, from the Hausdorff-Young inequality for $e^{-itH_0}$ we conclude
\begin{equation}
\label{eq:HY}
\|U_q \phi \|_{L^p} \le \|\hat\phi\|_{L^{p'}},
\end{equation}
where $p \in [2,\infty]$ and $p'= \frac p {p-1}$.
\begin{proof}

It is enough to show
\[
U_q(t)\phi(x) = \big[e^{-itH_0}\phi(x) + e^{-itH_0}(\phi * \rho_q)(-x)\big]x_-^0 + e^{-itH_0}(\phi * \tau_q)(x)x_+^0.
\]
From the definition of the propagator we have
\[
U_q(t)\phi(x) = \frac 1 {2\pi} \int_0^\infty \!\! \int e^{-it\lambda^2/2} \left( e_+(x,\lambda) \overline{e_+(y,\lambda)} + e_-(x,\lambda)\overline{e_-(y,\lambda)}\right) \phi(y) dy d\lambda,
\]
and so we must verify
\begin{align*}
\frac 1 {2\pi} \int_0^\infty e^{-it\lambda^2/2} &\left( e_+(x,\lambda)\int \overline{e_+(y,\lambda)}\phi(y)dy + e_-(x,\lambda)\int\overline{e_-(y,\lambda)} \phi(y) dy \right) d\lambda \\
&= \big[e^{-itH_0}\phi(x) + e^{-itH_0}(\phi * \rho_q)(-x)\big]x_-^0 + e^{-itH_0}(\phi * \tau_q)(x)x_+^0.
\end{align*}
We compute first
\begin{align*}
\int \overline{e_+(y,\lambda)}\phi(y)dy &= \int_{-\infty}^0 e^{-i\lambda y} \phi(y) d\lambda + \overline{r_q(\lambda)} \int_{-\infty}^0 e^{i\lambda y} \phi(y) d\lambda \\
&= \hat\phi(\lambda) + r_q(-\lambda) \hat\phi(-\lambda), \\
\int \overline{e_-(y,\lambda)}\phi(y)dy &= \overline{t_q(\lambda)}\int_{-\infty}^0 e^{i\lambda y} \phi(y) d\lambda = t_q(-\lambda) \hat\phi(-\lambda).
\end{align*}
We first verify the equation for positive $x$:
\begin{align*}
&\frac 1 {2\pi} \int_0^\infty e^{-it\lambda^2/2} \left( e_+(x,\lambda)\int \overline{e_+(y,\lambda)}\phi(y)dy + e_-(x,\lambda)\int\overline{e_-(y,\lambda)} \phi(y) dy \right) d\lambda \\
&=\frac 1 {2\pi} \int_0^\infty e^{-it\lambda^2/2} \left( e_+(x,\lambda)(\hat\phi(\lambda) + r_q(-\lambda) \hat\phi(-\lambda)) + e_-(x,\lambda) t_q(-\lambda) \hat\phi(-\lambda)\right) d\lambda  \\
&=\frac 1 {2\pi} \int_0^\infty e^{-it\lambda^2/2} 
\begin{aligned}[t]
\Big( & t_q(\lambda)e^{i\lambda x}(\hat\phi(\lambda) + r_q(-\lambda) \hat\phi(-\lambda)) \\
&+ (e^{-i\lambda x} + r_q(\lambda) e^{i\lambda x}) t_q(-\lambda) \hat\phi(-\lambda)\Big) d\lambda 
\end{aligned}
\intertext{At this stage we use $t_q(\lambda) r_q(-\lambda) + r_q(\lambda) t_q(-\lambda) = 2 \Re\left(t_q(\lambda) \overline{r_q(\lambda)}\right) = 0$:}
&=\frac 1 {2\pi} \int_0^\infty e^{-it\lambda^2/2} \left( t_q(\lambda)e^{i\lambda x}\hat\phi(\lambda) + e^{-i\lambda x}  t_q(-\lambda)\hat\phi(-\lambda)\right) d\lambda \\
&=\frac 1 {2\pi} \int_{-\infty}^\infty e^{-it\lambda^2/2}  t_q(\lambda) e^{i\lambda x} \hat\phi(\lambda) d\lambda = e^{-itH_0} (\tau_q * \phi) (x).
\end{align*}
The proof for negative $x$ is similar, except that it uses $r_q(\lambda) r_q(-\lambda) + t_q(\lambda) t_q(-\lambda) = |r_q(\lambda)|^2 + |t_q(\lambda)|^2=1$.
\end{proof}

We have two simple applications of Lemma \ref{p:lin}:  the Strichartz estimate (Proposition \ref{p:Str}) and the asymptotics of the linear flow $e^{itH_q}$ as $v \to \infty$ (Proposition \ref{p:as}).  We start with the Strichartz estimate, which will be used several times in the various approximation arguments of \S \ref{proof}.

\begin{prop}
\label{p:Str}
Suppose 
\begin{equation}
\label{eq:eq}
 i \partial_t u ( x , t ) + \tfrac{1}{2}\partial_{x}^2 u ( x, t ) 
- q \delta_0 ( x ) u ( x , t ) = f ( x , t ) \,, \ \ u ( x , 0 ) = \phi ( x ) 
\,.\end{equation}
Let the indices $p,r$, $\tilde p$, $\tilde r$ satisfy
$$2 \leq p, r \leq \infty \,, \ \ 1 \leq \tilde p , \tilde r \leq 2 \,, \ \
\frac 2 p + \frac 1 r = \frac 12 \,, \ \ \ \frac 2 {\tilde p }
+ \frac 1 {\tilde r} = \frac 52$$
and fix a time $T>0$. Then
\begin{equation}
\label{eq:Strneg}
\| u \|_{ L^p_{[0,T]} L^r_x } \leq c (\| \phi \|_{L^2} +T^\frac 1p \| P\phi\|_{L^r} + \|f\|_{L_{[0,T]}^{\tilde p} L_x^{\tilde r}}+T^{\frac 1p}\| Pf \|_{ L_{[0,T]}^1 L_x^r})
\end{equation}
The constant $c$ is independent of $q$ and $T$.  Moreover we can take $f(x,t) = g(t)\delta_0(x)$ and, on the right-hand side, replace $\| f \|_{ L_{[0,T]}^{\tilde p} L_x^{\tilde r} }$  with $\|g\|_{L_{[0,T]}^\frac{4}{3}}$ and replace $T^{\frac 1p}\| Pf \|_{ L_{[0,T]}^1 L_x^r}$ with $T^\frac{1}{p}|q|^{1-\frac{1}{r}}\|g\|_{L_{[0,T]}^1}$.
\end{prop}
\begin{proof}
This will follow from
\begin{equation}\label{eq:Struq}
\left\|U_q(t)\phi(x) + \int_0^t U_q(t-s)f(x,s)ds \right\|_{L^p_{[0,T]}L^r_x} \le c(\| \phi \|_{L^2} + \|f\|_{L^{\tilde p}_{[0,T]}L^{\tilde r}_x }).
\end{equation}
To prove this estimate, we observe that the case $q=0$ is the standard Euclidean Strichartz estimate (see \cite{KT} and \cite[Proposition 2.2]{HMZ}). The case $q \ne 0$ reduces to this case as follows. We write $\phi^- = x^0_- \phi$ and $\phi^+(x) = R x^0_+ \phi$ where again $R\phi(x) = \phi(-x)$. Note that $\phi = \phi^- + R \phi^+$, that $U_q(t) R = R U_q(t)$, and that $(Rf)*g = R(f*Rg)$. Now, from Lemma \ref{p:lin}, we have
\begin{align*} U_q \phi = &\big[U_0(t)\phi^- + U_0(t)(\phi^- * \rho_q)\big]x_-^0
+ U_0(t)(\phi^- * \tau_q)x_+^0 \\
&+ \big[U_0(t)\phi^+ + U_0(t)(\phi^+ * R\rho_q)\big]x_+^0
+ U_0(t)(\phi^+ * R\tau_q)x_-^0. \end{align*}
We must now show that $\|f^\pm * \sigma_q\|_{L^{\tilde p}_{[0,T]}L^{\tilde r}_x } \le c\|f\|_{L^{\tilde p}_{[0,T]}L^{\tilde r}_x}$, where $\sigma_q$ is either $\tau_q$ or $\rho_q$. This follows from applying Young's inequality to the spatial integral.

This completes the proof of (\ref{eq:Struq}). To obtain (\ref{eq:Strneg}), we observe that $\|P\phi(x)\|_{L^p_{[0,T]}L^r_x} = T^{1/p}\|P\phi(x)\|_{L^r_x}$ and $\left\|\int_0^t P f(x,s) ds \right\|_{L^p_{[0,T]}L^r_x} \le T^{1/p}\|Pf(x,s)\|_{L^1_{[0,T]}L^r_x}$. The first is immediate, and the second follows from the generalized Minkowski inequality:
\[
\left\|\int_0^t P f(x,s) ds \right\|_{L^p_{[0,T]}L^r_x} \le \left\|\int_0^t \|P f(x,s) \|_{L^r_x}ds\right\|_{L^p_{[0,T]}} \le T^{1/p}\int_0^T \|P f(x,s) \|_{L^r_x}ds.
\]
\end{proof}

We now turn to the large velocity asymptotics of the linear flow $e^{-itH_q}$.

\begin{prop}
\label{p:as}
Let $\theta(x)$ be a smooth function bounded, together with all of its derivatives, on $\mathbb{R}$.  Let $\phi\in \mathcal{S}(\mathbb{R})$, $v>0$, and suppose $\supp [\theta(x)\phi(x-x_0)] \subset (-\infty,0]$.  Then for $2|x_0|/v \leq t \leq 1$,
\begin{equation}
\label{E:as2}
e^{-itH_q}[e^{ixv}\phi(x-x_0)] =
\begin{aligned}[t]
& t(v)e^{-itH_0}[e^{ixv}\phi(x-x_0)] + r(v) e^{-itH_0}[e^{-ixv}\phi(-x-x_0)] \\
&+ e(x,t)
\end{aligned}
\end{equation}
where, for any $k\geq 0$,
$$\|e(x,t)\|_{L_x^2} \leq 
\begin{aligned}[t]
&\frac{1}{v}\|\partial_x [\theta(x)\phi(x-x_0)]\|_{L_x^2} +\frac{c_k}{(tv)^k} \| \la x \ra^k \phi(x) \|_{H_x^k}\\
&+4\|(1-\theta(x))\phi(x-x_0)\|_{L_x^2} + \|P[e^{ixv}\theta(x)\phi(x-x_0)]\|_{L_x^2}
\end{aligned}$$
\end{prop}

\begin{rem}
In \S \ref{proof}, Proposition \ref{p:as} will be applied with $\theta(x)$ a smooth cutoff to $x<0$, and $\phi(x)=\sech(x)$ with $x_0=-v^\eps \ll 0$.
\end{rem}

Before proving Proposition \ref{p:as}, we need the following
\begin{lem}
\label{l:as}
Let $\psi\in \mathcal{S}(\mathbb{R})$ with $\supp \psi \subset (-\infty,0]$.  Then
\begin{equation}
\label{E:as1}
\begin{aligned}[t]
U_q(t)[e^{ixv}\psi(x)](x) &= e^{-itH_0}[e^{ixv}\psi(x)](x)x_-^0 + t(v)e^{-itH_0}[e^{ixv}\psi(x)](x)x_+^0 \\
& \qquad + r(v)e^{-itH_0}[e^{-ixv}\psi(-x)](x)x_-^0 + e(x,t),
\end{aligned}
\end{equation}
where 
$$\|e(x,t)\|_{L_x^2} \leq \frac{1}{v}\|\partial_x \psi\|_{L^2}$$
uniformly in $t$.
\end{lem}
\begin{proof}
We apply (\ref{eq:prop}) with $\phi(x) = e^{ixv}\psi(x)$ to find that
\[
e(x,t) = e^{-itH_0}[\phi * (\rho_q - r_q(v)\delta_0)(-x)]x_-^0 + e^{-itH_0}[\phi * (\tau_q - t_q(v)\delta_0)(x)]x_+^0.
\]
We pass to the Fourier transform using Plancherel's theorem:
\[
\|e^{-itH_0} \varphi\|_{L^2} = c \|\hat\varphi\|_{L^2}.
\]
So that it suffices to verify
\[
\|\hat \psi(\lambda-v) (r_q(\lambda) - r_q(v))\|_{L^2} \le \frac c v \|\lambda\hat\psi(\lambda)\|_{L^2}, \]
\[
\|\hat \psi(\lambda-v) (t_q(\lambda) - t_q(v))\|_{L^2} \le \frac c v \|\lambda\hat\psi(\lambda)\|_{L^2}.
\]
We write out now
\[
r_q(\lambda) - r_q(v) = \frac{iq(\lambda - v)}{(i\lambda - q)(iv-q)}, \qquad t_q(\lambda) - t_q(v) = \frac{-iq(\lambda - v)}{(i\lambda - q)(iv-q)}.
\]
\[
|r_q(\lambda) - r_q(v)| = |t_q(\lambda) - t_q(v)| = \frac{|\lambda - v|}{\sqrt{\lambda^2/q^2 + 1}\sqrt{v^2 + q^2}} \le \frac {|\lambda - v|} v.
\]
We plug this in to obtain our estimate:
\[
\|\hat \psi(\lambda - v) (r_q(\lambda) - r_q(v))\|_{L^2} = \|\hat \psi(\lambda - v) (t_q(\lambda) - t_q(v))\|_{L^2} \le \frac 1 v \|\lambda \hat \psi(\lambda)\|_{L^2} \le \frac c v \|\D_x \psi\|_{L^2}.
\]
\end{proof}

Now we turn to the proof of Proposition \ref{p:as}.
\begin{proof}
We will prove (\ref{E:as2}) by showing that
\[
U_q(t)[e^{ixv}\phi(x-x_0)] = t(v)e^{-itH_0}[e^{ixv}\phi(x-x_0)] + r(v)e^{-itH_0}[e^{-ivx}\phi(-x-x_0)] + \tilde e(x,t),
\]
where, for any $k \ge 0$,
\[
\|\tilde e(x,t)\|_{L^2_x} \le \frac 1 v  \|\D_x[\theta(x)\phi(x-x_0)]\|_{L^2_x} + \frac {c_k}{(tv)^k} \| \langle x \rangle^k \phi(x)\|_{H^k_x} + 4 \|(1 - \theta(x))\phi(x-x_0)\|_{L^2_x}.
\]

We use (\ref{E:as1}) with $\psi(x) = \theta(x)\phi(x-x_0)$, which gives (using $e_1$ for the error term arising from the lemma),
\begin{align*}
U_q(t)[e^{ixv}\theta(x)\phi(x-x_0)](x) = & e^{-itH_0}[e^{ixv}\theta(x)\phi(x-x_0)](x)x_-^0 \\
&+ t(v)e^{-itH_0}[e^{ixv}\theta(x)\phi(x-x_0)](x)x_+^0 \\
&+ r(v)e^{-itH_0}[e^{-ixv}\theta(x)\phi(-x-x_0)](x)x_-^0 \\
&+ e_1(x,t).
\end{align*}
We rewrite this equation with $\theta$ omitted at the cost of an additional error term:
\begin{align*}
U_q(t)[e^{ixv}\phi(x-x_0)](x) = &e^{-itH_0}[e^{ixv}\phi(x-x_0)](x)x_-^0 \\
&+ t(v)e^{-itH_0}[e^{ixv}\phi(x-x_0)](x)x_+^0 \\
&+ r(v)e^{-itH_0}[e^{-ixv}\phi(-x-x_0)](x)x_-^0 + e_1(x,t) + e_2(x,t).
\end{align*}
Using the notation $f(x) = e^{ixv}(1 - \theta(x))\phi(x-x_0)$, this error term is given by
\[
e_2(x,t) = U_q(t)f(x) - e^{-itH_0}f(x)x_-^0 - t(v)e^{-itH_0}f(x)x_+^0 - r(v)e^{-itH_0}f(-x)x_-^0.
\]
Recall that $e^{-itH_0}$ is unitary, and that $\|U_q(t)\|_{L^2_x \to L^2_x} = 1$. This gives us a bound on $e_2$:
\[
\|e_2(x,t)\|_{L^2_x} \le 4 \|f\|_{L^2_x} = 4 \|(1 - \theta(x))\phi(x-x_0)\|_{L^2_x}.
\]

We have a bound on $e_1$ from the lemma, so combining everything we have, we see that to prove the proposition it remains only to prove
\begin{align*}
\indentalign \|e^{-itH_0}[e^{ixv}\phi(x-x_0)](x)x_-^0\|_{L^2_x} + \|t(v)e^{-itH_0}[e^{ixv}\theta(x)\phi(x-x_0)](x)x_-^0\|_{L^2_x} \\ 
&+ \|r(v)e^{-itH_0}[e^{-ixv}\phi(-x-x_0)](x)x_-^0\|_{L^2_x} \le \frac {c_k}{(tv)^k} \| \langle x \rangle^k \phi(x)\|_{H^k_x},
\end{align*}
for every $k \ge 0$. However, because $|t(v)| \le 1$ and $|r(v)| \le 1$, and because $e^{-itH_0} R = R e^{-itH_0}$, where $R\phi(x) = \phi(-x)$, it suffices to prove
\begin{equation}\label{suff}
\|e^{-itH_0}[e^{ixv}\phi(x-x_0)](x)x_-^0\|_{L^2_x} \le \frac {c_k}{(tv)^k} \| \langle x \rangle^k \phi(x)\|_{H^k_x}.
\end{equation}
We expand as follows, first using the definition of the propagator:
$$
e^{-itH_0}[e^{ixv}\phi(x-x_0)](x) = \frac 1 {2\pi} \int e^{ix\l - it\l^2/2} \int e^{-i\l y}e^{iyv}\phi(y-x_0)dyd\l.
$$
Here we change variables $y \mapsto y + x_0$ to simplify the Fourier transform.
\begin{align*}
&= \frac 1 {2\pi} \int e^{ix\l - it\l^2/2} e^{ix_0(v - \l)} \hat \phi(\l - v) d\l.
\intertext{Here we change variables $\l \mapsto \l + v$.}
&= \frac 1 {2\pi} e^{ixv}\int e^{ix\l - it(\l+v)^2/2} e^{-ix_0\l} \hat \phi(\l) d\l \\
&= \frac 1 {2\pi} e^{ixv}e^{-itv^2/2}\int e^{i\l(x-x_0 - tv)}e^{-it\l^2/2} \hat \phi(\l) d\l.
\intertext{After $k$ integrations by parts in $\l$, this becomes}
&= \frac 1 {2\pi} \frac {e^{ixv-itv^2/2}} {(i(x-x_0 - tv))^k} \int e^{i\l(x-x_0 - tv)} \D_\l^k\left(e^{-it\l^2/2} \hat \phi(\l)\right) d\l.
\end{align*}
By assumption we have $2|x_0|/v \le t \le 1$, so that $-x_0 -tv < 0$ and $|x-x_0 - tv| \ge |-x_0 - tv| \ge tv/2$ when $x < 0$. Hence
\begin{align*}
\left\| e^{-itH_0}[e^{ixv}\phi(x-x_0)](x)x^0_- \right\|_{L^2_x} &\le \frac {c_k}{(tv)^k} \left\| \int e^{i\l(x-x_0 - tv)} \D_\l^k\left(e^{-it\l^2/2} \hat \phi(\l)\right) d\l \right\|_{L^2_x} \\
&= \frac {c_k}{(tv)^k} \left\|\D_\l^k\left(e^{-it\l^2/2} \hat \phi(\l)\right) \right\|_{L^2_\l} \\
&\le \frac {c_k}{(tv)^k} \sum_{j=0}^k t^j\left\| \langle x \rangle^{k-j} \phi(x) \right\|_{H^j_x}.
\end{align*}
Since $t \le 1$, (\ref{suff}) follows.
\end{proof}

\section{Soliton scattering}
\label{proof}

In this section, we outline the proof of Theorem 1, the details of which are executed in \S \ref{S:phase1}--\ref{S:phase3}.  We recall the notation for operators from \S\ref{ros} and introduce short hand notation for the nonlinear flows:
\begin{itemize}
\item $H_0=-\frac{1}{2} \partial_x^2$.  The flow $e^{-itH_0}$ is termed the ``free linear flow''.
\medskip
\item $H_q = -\frac{1}{2} \partial_x^2+q\delta_0(x)$.  The flow $e^{-itH_q}$ is termed the ``perturbed linear flow''. We also use $U_q(t) \defeq e^{ - i t H_q } - e^{\frac{1}{2}itq^2}P$, the propagator corresponding to the continuous part of the spectrum of $H_q$.
\medskip
\item $\nlsq(t)\phi$, termed the ``perturbed nonlinear flow'' is the evolution of initial data $\phi(x)$ according to the equation $i\partial_tu + \tfrac{1}{2}\partial_x^2 u - q\delta_0(x)u + |u|^2u=0$.
\medskip
\item $\nlso(t)\phi$, termed the ``free nonlinear flow'' is the evolution of initial data $\phi(x)$ according to the equation $i\partial_th + \tfrac{1}{2}\partial_x^2 h + |h|^2h=0$.
\end{itemize}

From \S\ref{in} and the statement of Theorem 1, we recall the form of the initial condition, $u_0(x) = e^{ixv}\sech(x-x_0)$, $v \gtrsim 1$, $x_0 \leq -v^{\eps}$ where $0<\eps<1$ is fixed, and put $u(x,t) = \nlsq(t)u_0(x)$.  We begin by outlining the scheme, and will then supply the details.  In this section, the $\mathcal O$ notation always means $L_x^2$ difference, uniformly on the time interval specified, and up to a multiplicative factor that is independent of $q$, $v$, and $\eps$.

\noindent\textbf{Phase 1 (Pre-interaction)}.  Consider $0\leq t \leq t_1$, where $t_1 = |x_0|/v-v^{-(1-\eps)}$ so that $x_0+vt_1=-v^\eps$.  The soliton has not yet encountered the delta obstacle and propagates according to the free nonlinear flow. Indeed, there exists a small absolute constant $c$ such that $\la q \ra^3 e^{-cv^\eps} \le c$ implies
\begin{equation}
\label{E:approx1}
u(x,t) = e^{-itv^2/2}e^{it/2}e^{ixv}\sech(x-x_0-vt) + \mathcal{O}(|q|^\frac 32 \la q \ra^\frac 12e^{-v^\eps}), \quad 0\leq t\leq t_1.
\end{equation}
This is deduced as a consequence of Lemma \ref{L:approx1} in \S \ref{S:phase1} below.

\noindent\textbf{Phase 2 (Interaction)}.  Let $t_2 = t_1+v^{1-\eps}$, so that $x_0+t_2v = v^\eps$, and consider $t_1\leq t \leq t_2$.  The incident soliton, beginning at position $-v^\eps$, encounters the delta obstacle and splits into a transmitted component and a reflected component, which by $t=t_2$, are concentrated at positions $v^\eps$ and $-v^\eps$, respectively.  More precisely, there exists a small absolute constant $c$ such that if $v \la q \ra^3 \left(e^{\frac{-|q|v^\eps}2} + e^{-cv^\eps }\right) + v^{- \frac 23(1-\eps)}|q|^\frac 13 \le c$, then at the conclusion of this phase (at $t=t_2$),
\begin{equation}
\label{E:approx4}
u(x,t_2) = 
\begin{aligned}[t]
&t(v)e^{-it_2v^2/2}e^{it_2/2}e^{ixv}\sech(x-x_0-vt_2)\\
&+r(v)e^{-it_2v^2/2}e^{it_2/2}e^{-ixv}\sech(x+x_0+vt_2) \\
&+ \mathcal{O}(v^{-\frac 76 (1-\eps)}|q|^{\frac 13}) + \mathcal O(v^{-(1-\eps)}).
\end{aligned}
\end{equation}
This is proved as a consequence of Lemmas \ref{linapp}, \ref{L:approx2}, and \ref{L:dropint} in \S\ref{S:phase2}.

\noindent\textbf{Phase 3 (Post-interaction)}.  Let $t_3=t_2+ \eps\log v$, and consider $[t_2,t_3]$.  Suppose $|q|^{10}\la q \ra^3 v^{-14(1-2\eps)} + v \la q \ra^3 \left(e^{\frac{-|q|v^\eps}2} + e^{-cv^\eps }\right) \le c$ and $\la q\ra v^{-n} \le c_{\eps,n}$, where $c$ is a small absolute constant and $c_{\eps,n}$ is a small constant, dependent only on $\eps$ and $n$, which goes to zero as $\eps \to 0$ or $n \to \infty$. The transmitted and reflected waves essentially do not encounter the delta potential and propagate according the free nonlinear flow, 
\begin{equation}
\label{E:post}
u(x,t) = 
\begin{aligned}[t]
& e^{-itv^2/2}e^{it_2/2}e^{ixv}\nlso(t-t_2)[t(v)\sech(x)](x-x_0-tv) \\
&+ e^{-itv^2/2}e^{it_2/2}e^{-ixv}\nlso(t-t_2)[r(v)\sech(x)](x+x_0+tv) \\
&+\mathcal{O}(v^{-(1-\eps)}) + \mathcal{O}(|q|^\frac 13v^{-\frac 76(1-2\eps)}), \qquad t_2\leq t \leq t_3.
\end{aligned}
\end{equation}

Now we turn to the details. 

\section{Phase 1}
\label{S:phase1}

Let $u_1(x,t) = \nlso(t)u_0(x)$ and note that
$$u_1(x,t) = e^{-itv^2/2}e^{it/2}e^{ixv}\sech(x-x_0-tv)$$
Recall that $t_1= |x_0|/v-v^{-(1-\eps)}$, so that at the conclusion of Phase 1 (when $t=t_1$), the position of the soliton is $x_0+vt_1=-v^{\eps}$.    Recall that  $u(x,t)=\nlsq(t)u_0(x)$ and let $w= u-u_1$.  We will need the following perturbation lemma.

\begin{lem}
\label{L:approx1}
If $t_a < t_b \le t_1$, $t_b - t_a \le c_1 \le 1$, and
\[ \la q\ra^\frac 12\|w(x,t_a)\|_{L^2_x} + |q|^\frac 32 \la q \ra^\frac 12\|u_1(0,t)\|_{L^1_{[t_a,t_b]}} \le c_2|q|^{-\frac 14},\]
then
\[\|w\|_{L^p_{[t_a,t_b]}L^r_x} \le c_3 \left(\la q\ra^\frac 12\|w(x,t_a)\|_{L^2_x} + |q|^\frac 32 \la q \ra^\frac 12\|u_1(0,t)\|_{L^1_{[t_a,t_b]}}\right),\]
where $(p,r)$ is either $(\infty,2)$ or $(4,\infty)$, and the constants $c_1, c_2$, and $c_3$ are independent of the parameters $\eps, v,$ and $q$.\end{lem}

Before proving this lemma, we show how the phase 1 estimate follows from it. Let $k\geq 0$ be the integer such that $kc_1 \leq t_1 < (k+1)c_1$.  (Note that $k=0$ if the soliton starts within a distance $v$ of the origin, i.e. $-v-v^\eps< x_0\leq -v^\eps$, and the inductive analysis below is skipped.) Apply Lemma \ref{L:approx1} with $t_a=0$, $t_b=c_1$ to obtain (since $w(\cdot,0)=0$)
\[\|w\|_{L_{[0,c_1]}^\infty L_x^2} \leq c_3|q|^\frac 32 \la q \ra^\frac 12\|u_1(0,t)\|_{L_{[0,c_1]}^\infty}\leq c_3|q|^\frac 32 \la q \ra^\frac 12\sech(x_0+vc_1).\]
Apply Lemma \ref{L:approx1} again with $t_a=c_1$, $t_b=2c_1$ to obtain
\begin{align*}
\|w\|_{L_{[c_1,2c_1]}^\infty L_x^2} &\leq c_3(\la q \ra^\frac12\|w(\cdot,c_1)\|_{L_x^2}+ |q|^\frac 32 \la q \ra^\frac 12\|u_1(0,t)\|_{L_{[c_1,2c_1]}^\infty}) \\
&\leq c_3^2|q|^\frac 32 \la q \ra\sech(x_0+vc_1)+c_3^1|q|^\frac 32 \la q \ra^\frac 12\sech(x_0+2vc_1).
\end{align*}
We continue inductively up to step $k$, and then collect all $k$ estimates to obtain the following bound on the time interval $[0,kc_1]$:
\begin{align*}
\|w\|_{L_{[0,kc_1]}^\infty L_x^2} &\leq c_3|q|^\frac 32 \sum_{j=1}^k c_3^{k-j} \la q\ra^\frac{k-j}2\sech(x_0+jvc_1).
\intertext{We use here the estimate $\sech \alpha \leq 2e^{-|\alpha|}$:}
& \le c_3^k|q|^\frac32 \la q \ra^\frac {k-1}2 e^{x_0+c_1v} \sum_{j=0}^{k-1}c_3^{-j}\la q\ra^{-\frac{j}{2}}e^{jvc_1}.
\intertext{We introduce here the assumption that $c_3^{-1}\la q\ra^{-\frac{1}{2}}e^{vc_1} \ge 2$, which allows us to estimate the geometric series by twice its final term, giving}
&\le  c |q|^\frac32 e^{x_0+c_1vk}.
\end{align*}
Finally, applying Lemma \ref{L:approx1} on $[k c_1,t_1]$,
\[ \|w\|_{L_{[0,t_1]}^\infty L_x^2} \leq c\left(|q|^\frac32\la q \ra^\frac 12 e^{x_0+c_1vk} + |q|^\frac 32 \la q \ra^\frac 12 \sech(x_0+t_1v)\right) \leq c|q|^\frac32\la q \ra^\frac 12e^{-v^\eps},\]
where we have used $x_0+c_1vk \le x_0+t_1v = - v^\eps$. The constant $c$ here is still independent of $q$, $v$, and $\eps$. As a consequence, \eqref{E:approx1} follows. So long as this last line is bounded by $c\la q\ra^\frac 14$ the repeated applications of Lemma \ref{L:approx1} are justified.

Now we prove Lemma \ref{L:approx1}:
\begin{proof}
We begin by writing a differential equation for $w$ in terms of $u_1$:
\begin{align*}
i\D_t w + \frac 1 2 \D_x^2 w  - q\delta_0(x) w &= u_1|u_1|^2 - (w+u_1)|w + u_1|^2 + q \delta_0(x) u_1 \\
&= - w|w|^2 - 2 u_1 |w|^2 - \ou w^2 - 2|u_1|^2w - u_1^2\ow + q \delta_0(x) u_1.
\end{align*}
We rewrite this as an integral equation in terms of the perturbed linear propagator $e^{-iH_q t}=U_q(t) + e^{itq^2/2}P$, regarding the right hand side as a forcing term.
\begin{align*}
&w(x,t) = \left(U_q(t-t_a) + e^{i(t-t_a)q^2/2}P\right)w(x,t_a) \\
&\qquad + \int_{t_a}^t U_q(t-s)\left(- w|w|^2 - 2 u_1 |w|^2 - \ou w^2 - 2|u_1|^2w - u_1^2\ow + q \delta_0(x) u_1\right) ds \\
&\qquad + \int_{t_a}^t e^{i(t-s)q^2/2}P\left(- w|w|^2 - 2 u_1 |w|^2 - \ou w^2 - 2|u_1|^2w - u_1^2\ow + q \delta_0(x) u_1\right) ds \\
&=  \textrm{I}  +  \textrm{II}  +  \textrm{III}.
\end{align*}
We define $\|w\|_X = \|w\|_{L^\infty_{[t_a,t_b]}L^2_x} + \|w\|_{L^4_{[t_a,t_b]}L^\infty_x}$, and then proceed to estimate the $X$ norm of the right hand side term by term. In what follows, $(p,r)$ denotes either $(\infty,2)$ or $(4,\infty)$.

I. We observe from the Strichartz estimate that
\[\|U_q(t-t_a) w\|_X \le c \|w(x,t_a)\|_{L^2_x}.\]
Next, using the H\"older estimate for $P$, we have
\[\|P w(x,t_a)\|_{L^r_x} = (t_b-t_a)^{1/p}\|P w(x,t_a)\|_{L^r_x} \le c (t_b-t_a)^{1/p}|q|^{1/2 - 1/r}\|w(x,t_a)\|_{L^2_x}.\]
Taken together, and using $t_b - t_a \le c_1$, these bounds can be written as
\[\left\|\left(U_q(t-t_a) + e^{i(t-t_a)q^2/2}P\right)w(x,t_a)\right\|_X \le c \la q\ra^{1/2}\|w(x,t_a)\|_{L^2_x}.\]
II. For the terms involving $U_q$ we will use the Strichartz estimate, which tells us that
\[\left\|\int_{t_a}^t U_q(t-s)f(x,s)ds \right\|_X \le c\|f\|_{L^{\tilde p}_{[t_a,t_b]}L^{\tilde r}_x },\]
whenever $2/\tilde p + 1 / \tilde r = 5/2$. The first term, $w|w|^2$, is cubic, and we will use $\tilde p = 1$ and $\tilde r = 2$:
$$
\left\|\int_{t_a}^t U_q(t-s)w|w|^2(x,s)ds \right\|_X \le  c \|w^3\|_{L^1_{[t_a,t_b]}L^2_x}.
$$
We first pass to the $L^\infty_x$ norm for two of the factors.
\begin{align*}
&\le c \left\|\|w\|_{L^2_x}(t)\|w^2\|_{L^\infty_x}(t)\right\|_{L^1_{[t_a,t_b]}}.
\intertext{And then pass to the $L^\infty_{[t_a,t_b]}$ norm for the other factor.}
&\le c \|w\|_{L^\infty_{[t_a,t_b]}L^2_x}\|w^2\|_{L^1_{[t_a,t_b]}L^\infty_x} = c \|w\|_{L^\infty_{[t_a,t_b]}L^2_x}\|w\|^2_{L^2_{[t_a,t_b]}L^\infty_x}.
\intertext{Finally we use the boundedness of $t_b-t_a$ to pass to the $X$ norm.}
&\le c (t_b-t_a)^\frac 12 \|w\|^3_X.
\end{align*}
The quadratic and linear terms follow the same pattern. Observe that the distinction between $w$ and $\ow$, and between $u_1$ and $\ou$, does not play a role here:
\begin{align*} \left\|\int_{t_a}^t U_q(t-s)u_1|w|^2(x,s)ds \right\|_X &\le c(t_b-t_a)^\frac 12\|u_1\|_X\|w\|^2_X \le c(t_b-t_a)^\frac 12\|w\|^2_X \\
\left\|\int_{t_a}^t U_q(t-s)w|u_1|^2(x,s)ds \right\|_X &\le c(t_b-t_a)^\frac 12\|w\|_{L^6_{[t_a,t_b]}L^6_x}. \end{align*}
For the delta term we have no flexibility in the choice of exponents.
\begin{align*}
|q|\left\|\int_{t_a}^t U_q(t-s)\delta_0(x)u_1(0,s)ds \right\|_X \le c |q| \|u_1(0,t)\|_{L^{4/3}_{[t_a,t_b]}}.
\end{align*}

III. The $P$ terms we will similarly estimate one by one:
$$
\bigg\|\int_{t_a}^t e^{i(t-s)q^2/2}P\big(- w|w|^2 - 2 u_1 |w|^2 - \ou w^2 - 2|u_1|^2w - u_1^2\ow + q \delta_0(x) u_1\big)ds\bigg\|_{{L^1_{[t_a,t_b]}L^r_x}} $$
$$\le 
\begin{aligned}[t]
c \Big( &\|P |w|^3\|_{L^1_{[t_a,t_b]}L^r_x} + \|P u_1 |w|^2\|_{L^1_{[t_a,t_b]}L^r_x}   \\
&+\| P |u_1|^2w \|_{L^1_{[t_a,t_b]}L^r_x} + |q| \| P \delta_0(x)u_1(0,t) \|_{L^1_{[t_a,t_b]}L^r_x}\Big).
\end{aligned}
$$
Here we used a generalized Minkowski inequality to pass the norm through the integral, just as we did in the proof of the Strichartz estimate. Note that the constant $c$ in the second line depends on $p$, but since $p$ only takes two different values we will not make this dependence explicit in our notation. For the delta term as before we have no flexibility in the choice of $L^{r_1}$ norm:
\[|q| \| P \delta_0(x)u_1(0,t) \|_{L^1_{[t_a,t_b]}L^r_x} \le c |q|^{2 - 1/r} \|u_1(0,t)\|_{L^1_{[t_a,t_b]}} \le c |q|^\frac32\la q\ra^\frac12 \|u_1(0,t)\|_{L^1_{[t_a,t_b]}}.\]
For the cubic term in $w$ we proceed using the same H\"older estimate for $P$. Here we use $r_1 = 2$, giving this term a factor no worse than $\la q\ra^\frac12$:
\[\|P |w|^3\|_{L^1_{[t_a,t_b]}L^r_x} \le c\la q\ra^\frac 12 \|w^3\|_{L^1_{[t_a,t_b]}L^2_x} = c\la q\ra^\frac 12 \|w\|^3_X.\]
The last step is the same as that in the $w^3$ term in II above.
We now estimate the quadratic and linear terms in $w$. We have
\begin{align*} \| P u_1w^2\|_{L^1_{[t_a,t_b]}L^r_x} &\le c\|w^2\|_{L^1_{[t_a,t_b]}L^r_x} = c\|w\|^2_{L^2_{[t_a,t_b]}L^{2r}_x} \le c \|w\|^2_X \\
\| P u_1^2w\|_{L^1_{[t_a,t_b]}L^r_x} &\le c \|w\|_{L^1_{[t_a,t_b]}L^r_x} \le c (t_b - t_a)^\frac 34 \|w\|_X. \end{align*}
When $r=\infty$ this last step is achieved by passing from $L^2_{[t_a,t_b]}$ to $L^4_{[t_a,t_b]}$ and using the boundedness of the time interval. When $r=2$ we interpolate using H\"older's inequality between $L^2_x$ and $L^\infty_x$.

Having estimated each of the terms individually, we combine our results. We use $t_b - t_a \le c_1 \le 1$ to pass from lower norms in time to higher ones, only tracking the power of $(t_b-t_a)$ for the linear term in $w$.
\begin{align*}
\|w\|_X \le c \bigg( &\la q\ra^\frac 12\|w(x,t_a)\|_{L^2_x} + |q|^\frac 32 \la q \ra^\frac 12\|u_1(0,t)\|_{L^1_{[t_a,t_b]}} \\
&+ (t_b-t_a)^{1/2}\|w\|_X + \|w\|_X^2 + |q|^{1/3} \|w\|^3_X\bigg).
\end{align*}
We now take $c_1$ sufficiently small (recall that $t_b - t_a \le c_1$) so that the linear term in $\|w\|_X$ can be absorbed into the left hand side:
\[\|w\|_X \le c\left(\la q\ra^\frac 12\|w(x,t_a)\|_{L^2_x} + |q|^\frac 32 \la q \ra^\frac 12\|u_1(0,t)\|_{L^1_{[t_a,t_b]}}+  \|w\|^2_X + |q|^{1/3} \|w\|^3_X\right),\]
with a slightly worse leading constant $c$. We rewrite this inequality schematically using $x = \|w\|_X$:
\[0 \le A - x + Bx^2 + Cx^3.\]
We now consider $\|w\|_{X(t')}$ for $\|w\|_{X(t')} \stackrel{\textrm{def}}= \|w\|_{L^\infty_{[t_a,t']}L^2_x} + \|w\|_{L^4_{[t_a,t']}L^\infty_x}$ for $t' \in [t_a,t_b]$. This is a continuous function of $t'$, and, for each $t' \in [t_a,t_b]$, $\|w\|_{X(t')}$ obeys the above inequality. Therefore if we find a positive value $x_0$ for which the inequality does not hold, we will be able to conclude that $\|w\|_{X(t')} < x_0$ for every $t' \in [t_a,t_b]$, and hence also that $\|w\|_X < x_0$.

We will use $x_0 = 2A$, and arrange $A$, $B$ and $C$ so that this gives a negative right hand side. In fact, we have
\[A - 2A + 4BA^2 + 8CA^3 = -A + A(4BA + 8CA^2).\]
To make this negative we impose $4BA \le \frac 1 4$ and $8CA^2 \le \frac 1 4$. We thus obtain $x \le 2A$, or, in the language of $\|w\|_X$,
\[\|w\|_X \le c \left(\la q\ra^\frac 12\|w(x,t_a)\|_{L^2_x} + |q|^\frac 32 \la q \ra^\frac 12\|u_1(0,t)\|_{L^1_{[t_a,t_b]}}\right),\]
provided that $\la q\ra^\frac 12\|w(x,t_a)\|_{L^2_x} + |q|^\frac 32 \la q \ra^\frac 12\|u_1(0,t)\|_{L^1_{[t_a,t_b]}} \le c|q|^{-1/6}$.
\end{proof}

\section{Phase 2}
\label{S:phase2}

We begin with a succession of three lemmas stating that the free nonlinear flow is approximated by the free linear flow, and that the perturbed nonlinear flow is approximated by the perturbed linear flow. The first lemma states that the nonlinear flows are well approximated by the corresponding linear flows, the second gives a better approximation by adding a cubic correction term, and the third shows that the improvement is retained even if the cubic term is omitted. In other words, we `add and subtract' the cubic correction. Our estimates are consequences of the corresponding Strichartz estimates (Proposition \ref{p:Str}).  Crucially, the hypotheses and estimates of this lemma depend only on the $L^2$ norm of the initial data $\phi$.  Below, the lemmas are applied with $\phi(x)=u(x,t_1)$, and $\|u(x,t_1)\|_{L_x^2}=\|u_0\|_{L^2}$ is independent of $v$; thus $v$ does not enter adversely into the analysis. We first state the lemmas and show how they are applied, deferring the proofs to the end of the section.

\begin{lem}
\label{linapp}
Let $\phi\in L^2$ and $0<t_b$.  If $(t_b^{1/2} + t_b^{2/3} |q|^{1/3}) \le c_1 (\|\phi\|_{L^2} + t_b^{1/6}\|P\phi\|_{L^6})^{-2}$, then
\begin{equation}
\|\nlsq(t)\phi - e^{-itH_q}\phi\|_{L^P_{[0,t_b]}L^R_x} \le c_2 t_b^{1/2}(1 + t_b^{1/P}|q|^{2/P})(\|\phi\|_{L^2} + t_b^{1/6}\|P\phi\|_{L^6})^3,
\end{equation}
where $P$ and $R$ satisfy $\frac 2 {P} + \frac 1 {R} = \frac 1 2$, and $c_1$ and $c_2$ depend only on the constant appearing in the Strichartz estimates. We alert the reader that in our notation $P$ is used both as a Strichartz exponent and as the bound state projection \eqref{E:Pdef}.
\end{lem}

\begin{lem}
\label{L:approx2}
Under the same hypotheses as the previous lemma,
\begin{align}
\label{E: approx2}
\|\nlsq(t)\phi - g\|_{L^\infty_{[0,t_b]}L^2_x} \le c \bigg[&t_b^2\left(1 + t_b^{1/6}|q|^{1/3}\right)^3 (\|\phi\|_{L^2} + t_b^{1/6}\|P\phi\|_{L^6})^9 \\
\nonumber &+ t_b^{3/2}\left(1 + t_b^{1/6}|q|^{1/3}\right)^2(\|\phi\|_{L^2} + t_b^{1/6}\|P\phi\|_{L^6})^6\|\phi\|_{L^2} \\
\nonumber &+ t_b \left(1 + t_b^{1/6}|q|^{1/3}\right)(\|\phi\|_{L^2} + t_b^{1/6}\|P\phi\|_{L^6})^3\|\phi\|^2_{L^2}\bigg],
\end{align}
where
\[
g(t) = e^{-itH_q}\phi + \int_0^t e^{-i(t-s)H_q}|e^{-isH_q}\phi|^2e^{-isH_q}\phi ds.
\]
\end{lem}

\begin{lem}
\label{L:dropint}
For $t_1 < t_2$ and $\phi = u(x,t_1)$, we have
\begin{align}\label{E:dropint}
\indentalign \bigg\| \int_{t_1}^{t_2} e^{-i(t_2-s)H_q}|e^{-isH_q}\phi|^2e^{-isH_q}\phi ds \bigg\|_{L^2_x} \\
\nonumber &\leq c \left[(t_2-t_1) + (t_2-t_1)^{\frac 12}\left(\la q \ra^\frac 32 |q|^\frac 32 e^{-v^\eps} + \la q \ra^2 |q|^3 e^{-2v^\eps} + \la q \ra^\frac 52 |q|^\frac 92 e^{-3v^\eps}\right)\right],
\end{align}
where $c$ is independent of the parameters of the problem.
\end{lem}

In order to apply these lemmas, we need to estimate $\|P\phi\|_{L_x^6}$.  As before, let $\phi_1(x) = e^{-itv^2/2}e^{it/2}e^{ixv}\sech(x-x_0-vt_1)$ and $\phi_2=\phi-\phi_1$.  It suffices to estimate $\|P\phi_1\|_{L_x^6}$ and $\|P\phi_2\|_{L_x^6}$.  By \eqref{E:approx1},
\[\|\phi_2\|_{L_x^2} \leq  c \la q \ra^\frac 32 |q|^\frac 12 e^{-v^\eps},\]
and thus, by \eqref{E:PHolder}
\[\| P\phi_2\|_{L_x^6} \leq   c \la q \ra^\frac 32 |q|^\frac 56 e^{-v^\eps}. \]
On the other hand, a direct computation gives 
\begin{align*}
|P\phi_1(x)| &\le |q| e^{q|x|} \int e^{q|y|} \sech (y - x_0 - vt_1) dy \\
&= |q| e^{q|x|} 
\begin{aligned}[t]
\Big( &\int_{-\infty}^{\frac {x_0+vt_1} 2} e^{q|y|} \sech (y - x_0 - vt_1) dy \\
&+ \int_{\frac {x_0 + vt_1}2}^\infty e^{q|y|} \sech (y - x_0 - vt_1) dy\Big)
\end{aligned}
\end{align*}
In the first integral we use the fact that $e^{q|y|}$ is uniformly small in the region of integration. For the second we use the fact that $e^{q|y|}$ is bounded by 1 and that $\sech \alpha \le 2 e^{-\alpha}$.
$$
|P\phi_1(x)| \le c|q| e^{q|x|} \left(e^{\frac{|q|}2(x_0+vt_1)} + e^{\frac 12 (x_0 + vt_1)}\right) \le c|q| e^{q|x|} \left(e^{\frac{-|q|v^\eps}2} + e^{\frac {-v^\eps} 2}\right).
$$
This implies that
\[\|P\phi_1\|_{L^6_x} \le c|q|^\frac 56 \left(e^{\frac{-|q|v^\eps}2} + e^{\frac {-v^\eps} 2}\right)\]
Combining,
\begin{equation}\label{E:pphi6}\|P\phi\|_{L_x^6} \le c \la q \ra^\frac 32 |q|^\frac 56 \left(e^{-\frac 12v^\eps} + e^{\frac q2 v^\eps}\right).\end{equation}

Set $t_2=t_1+2v^{1-\eps}$ and $\phi(x) = u(x,t_1)$. We now give an interpretation of our three lemmas under the assumption $v \la q \ra^3 \left(e^{\frac{-|q|v^\eps}2} + e^{\frac {-v^\eps} 2}\right) + v^{- \frac 23(1-\eps)}|q|^\frac 13 \le 1$. This makes \eqref{E: approx2} into
\[ \|u(x,t_2) - g(t_2)\|_{L^2_x} \le c \left(v^{-1} + v^{-\frac 76 (1-\eps)}|q|^{\frac 13}\right), \]
and \eqref{E:dropint} into
\[\bigg\|g(t_2) - e^{-i(t_2-t_1)H_q}u(x,t_1)\bigg\|_{L^2_x} \le c v^{-(1-\eps)}. \]
Overall this amounts to
\begin{align*}
u(\cdot,t_2) &= \nlsq(t_2-t_1)[u(\cdot,t_1)] \\
&= 
\begin{aligned}[t]
&e^{-i(t_2-t_1)H_q}[u(\cdot,t_1)]+ \mathcal{O}(v^{-\frac 76 (1-\eps)}|q|^{\frac 13}) + \mathcal O(v^{-(1-\eps)}).
\end{aligned}
\end{align*}
By combining this with \eqref{E:approx1} we find that under our new assumption the errors from this phase are strictly larger, giving
\begin{equation}
\label{E:approx2}
u(\cdot,t_2) = 
\begin{aligned}[t]
&e^{-it_1v^2/2}e^{it_1/2}e^{-i(t_2-t_1)H_q}[e^{ixv}\sech(x-x_0-t_1v)] \\
&+ \mathcal{O}(v^{-\frac 76 (1-\eps)}|q|^{\frac 13}) + \mathcal O(v^{-(1-\eps)}).
\end{aligned}\end{equation}

By Proposition \ref{p:as} with $\theta(x)=1$ for $x\leq -1$ and $\theta(x)=0$ for $x\geq 0$, $\phi(x) = \sech(x)$, and $x_0$ replaced by $x_0+t_1v$,
\begin{equation}
\label{E:approx3}
\begin{aligned}
\indentalign e^{-i(t_2-t_1)H_q}[e^{ixv}\sech(x-x_0-vt_1)](x) \\
&=
\begin{aligned}[t]
&t ( v ) e^{-i(t_2-t_1)H_0}[e^{ixv}\sech(x-x_0-vt_1)](x) \\
&+ r ( v )  e^{-i(t_2-t_1)H_0}[e^{-ixv}\sech(x+x_0+vt_1)](x)  \\
&+ {\mathcal O} (v^{-1}),
\end{aligned}
\end{aligned}
\end{equation}
where we have used the assumption that $e^{-v^\eps} \le v^{-1}$. Now we use \eqref{E: approx2p} and \eqref{E:dropintp} to approximate $e^{-itH_0}$ by $\nlso$, picking up an error of ${\mathcal O}(v^{1-\eps})$. By combining with \eqref{E:approx2} and \eqref{E:approx3}, we obtain
\[
u(\cdot,t) =
\begin{aligned}[t]
&t(v)e^{-it_1v^2/2}e^{it_1/2}\nlso(t_2-t_1)[e^{ixv}\sech(x-x_0-vt_1)](x) \\
&+r(v)e^{-it_1v^2/2}e^{it_1/2}\nlso(t_2-t_1)[e^{-ixv}\sech(x+x_0+vt_1)](x)\\
&+ \mathcal{O}(v^{-\frac 76 (1-\eps)}|q|^{\frac 13}) + \mathcal O(v^{-(1-\eps)}).
\end{aligned}
\]
By noting that
\begin{align*}
\indentalign \nlso(t_2-t_1)[e^{ixv}\sech(x-x_0-t_1v)] \\
&= e^{-i(t_2-t_1)v^2/2}e^{i(t_2-t_1)/2}e^{ixv}\sech(x-x_0-t_2v),
\end{align*}
and
\begin{align*}
\indentalign \nlso(t_2-t_1)[e^{-ixv}\sech(x+x_0+t_1v)] \\
&= e^{-i(t_2-t_1)v^2/2}e^{i(t_2-t_1)/2}e^{-ixv}\sech(x+x_0+t_2v),
\end{align*}
we obtain \eqref{E:approx4}.

Now we prove Lemma \ref{linapp}:
\begin{proof}
Let $h(t) = \nlsq(t)\phi$, so that
\[
i\D_t h + \frac 1 2 \D_x^2 h - q \delta_0(x)h + |h|^2h = 0.
\]
Where in Phase 1 we used $L^4_tL^\infty_x$ as an auxiliary Strichartz norm, here we will use $L^6_tL^6_x$. We introduce the notation $\|h\|_{X'} = \|h\|_{L^\infty_{[0,t_b]}L^2_x} + \|h\|_{L^6_{[0,t_b]}L^6_x}$, and apply the Strichartz estimate
\[
\|h\|_{L^p_{[0,t_b]}L^r_x} \le c(\| \phi \|_{L^2} + t_b^{1/p} \| P \phi \|_{L^r} + \||h|^2h\|_{L^{\tilde p}_{[0,t_b]}L^{\tilde r}_x }+ t_b^{1/p} \|P |h|^2h\|_{L^1_{[0,t_b]}L^r_x})
\]
once with $(p,r) = (\infty, 2)$ and $(\tilde p, \tilde r) = (6/5,6/5)$, and once with $(p,r) = (6,6)$ and $(\tilde p, \tilde r) = (6/5,6/5)$.
We observe that H\"older's inequality implies that $\|f\|^3_{L^p} \le \|f\|_{L^{p_1}}\|f\|_{L^{p_2}}\|f\|_{L^{p_3}}$ provided $\frac 1 p = \frac 1 {3p_1} + \frac 1 {3p_2} + \frac 1 {3p_3}$. This gives us
\begin{align*}
\|h\|^3_{L^{18/5}_{[0,t_b]}L^{18/5}_x} &\le \|h\|^2_{L^6_{[0,t_b]}L^6_x} \|h\|_{L^2_{[0,t_b]}L^2_x} \le c t_b^{1/2} \|h\|^3_{X'}.
\intertext{We also have}
\|P |h|^2h\|_{L^1_{[0,t_b]}L^2_x} &\le \|h\|^3_{L^3_{[0,t_b]}L^6_x} \le t_b^{1/2} \|h\|_{X'}^3, \\
t_b^{1/6}\|P |h|^2h\|_{L^1_{[0,t_b]}L^6_x} &\le c t_b^{2/3}|q|^{1/3} \|h\|_{X'}^3,
\end{align*}
yielding
\[
\|h\|_{X'} \le c(\|\phi\|_{L^2} + \|P\phi\|_{L^2} + t_b^{1/6}\|P\phi\|_{L^6} + t_b^{1/2} \|h\|^3_{X'} + t_b^{2/3} |q|^{1/3} \|h\|_{X'}^3).
\]
We then use the fact that $P$ is a projection on $L^2$ to write
\[
\|h\|_{X'} \le c(\|\phi\|_{L^2} + t_b^{1/6}\|P\phi\|_{L^6} + t_b^{1/2} \|h\|^3_{X'} + t_b^{2/3} |q|^{1/3} \|h\|_{X'}^3).
\]
Using, as in Phase 1, the continuity of $\|h\|_{{X'}(t_b)}$, we conclude that
\[
\|h\|_{X'} \le 2c(\|\phi\|_{L^2} + t_b^{1/6}\|P\phi\|_{L^6}),
\]
so long as $8c^2(t_b^{1/2} + t_b^{2/3} |q|^{1/3})(\|\phi\|_{L^2} + t_b^{1/6}\|P\phi\|_{L^6})^2 \le 1$.

We now apply the Strichartz estimate to $u(t) = h(t) - e^{-itH_q}\phi$, observing that the initial condition is zero and the effective forcing term $-|h|^2h$, to get
\begin{align}
\|h(t) - e^{-itH_q}\phi\|_{L^P_{[0,t_b]}L^R_x} &\le c \||h|^2h\|_{L^{6/5}_{[0,t_b]}L^{6/5}_x} + t_b^{1/P}\|P |h|^2h\|_{L^1_{[0,t_b]}L^R_x}\\
\nonumber&\le ct_b^{1/2}\|h\|_{X'}^3 + ct_b^{1/P}|q|^{1/2 - 1/R}\|h^3\|_{L^1_{[0,t_b]}L^2_x})  \\
\nonumber&\le c(t_b^{1/2} + ct_b^{1/P + 1/2}|q|^{2/P})\|h\|_{X'}^3 \\
\nonumber&\le ct_b^{1/2}(1 + t_b^{1/P}|q|^{2/P})(\|\phi\|_{L^2} + t_b^{1/6}\|P\phi\|_{L^6})^3.
\end{align}
\end{proof}

Now we prove Lemma \ref{L:approx2}:

\begin{proof}
A direct calculation shows
\[
h(t) - g(t)  = \int_0^t e^{-i(t-s)H_q}\left(|h(s)|^2h(s) - |e^{-isH_q}\phi|^2e^{-isH_q}\phi\right)ds.
\]
The Strichartz estimate gives us in this case
\begin{align*}
\|h - g\|_{L^\infty_{[0,t_b]}L^2_x} &\le \||h|^2h - |e^{-isH_q}\phi|^2e^{-isH_q}\phi\|_{L^{6/5}_{[0,t_b]}L^{6/5}_x} \\
&\qquad\qquad+ \|P\left(|h|^2h - |e^{-isH_q}\phi|^2e^{-isH_q}\phi\right)\|_{L^1_{[0,t_b]}L^2_x} \\
&= \qquad \textrm{I} \qquad + \qquad \textrm{II}.
\end{align*}
We introduce the notation $w(t) = h(t) - e^{-itH_q}\phi$, and use this to rewrite our difference of cubes:
\[
|h|^2h - |e^{-isH_q}\phi|^2e^{-isH_q}\phi = 
\begin{aligned}[t]
&w|w|^2 + 2 e^{-isH_q}\phi |w|^2 + e^{isH_q}\overline{\phi} w^2 \\
&+ 2|e^{-isH_q}\phi|^2w + \left(e^{-isH_q}\phi\right)^2\ow.
\end{aligned}
\]
We proceed term by term, using H\"older estimates similar to the ones in the previous lemma and in Phase 1. Our goal is to obtain Strichartz norms of $w$, so that we can apply Lemma \ref{linapp}.

I. We have, for the cubic term,
\begin{align*}
\|w^3\|_{L^{6/5}_{[0,t_b]}L^{6/5}_x} &\le c t_b^{1/2}\|w\|_{L^\infty_{[t_a,t_b]}L^2_x}\|w\|^2_{L^6_{[t_a,t_b]}L^6_x} \\
&\le c t_b^2(1 + t_b^{1/6}|q|^{1/3})^2 (\|\phi\|_{L^2} + t_b^{1/6}\|P\phi\|_{L^6})^9.
\intertext{For the first inequality we used H\"older, and for the second Lemma \ref{linapp}. Next we treat the quadratic and linear terms using the same strategy (observe that as before we ignore complex conjugates):}
\|e^{-isH_q}\phi |w|^2\|_{L^{6/5}_{[0,t_b]}L^{6/5}_x} &\le c t_b^{1/2}\|w\|_{L^\infty_{[t_a,t_b]}L^2_x}\|w\|_{L^6_{[t_a,t_b]}L^6_x}\|e^{-isH_q}\phi\|_{L^6_{[t_a,t_b]}L^6_x} \\
&\le c t_b^{3/2}(1 + t_b^{1/6}|q|^{1/3})(\|\phi\|_{L^2} + t_b^{1/6}\|P\phi\|_{L^6})^6\|\phi\|_{L^2}. \\
\||e^{-isH_q}\phi|^2 w\|_{L^{6/5}_{[0,t_b]}L^{6/5}_x} &\le c t_b^{1/2}\|w\|_{L^\infty_{[t_a,t_b]}L^2_x}\|e^{-isH_q}\phi\|^2_{L^6_{[t_a,t_b]}L^6_x} \\
&\le c t_b (\|\phi\|_{L^2} + t_b^{1/6}\|P\phi\|_{L^6})^3\|\phi\|^2_{L^2}.
\end{align*}

II. In this case we have
\begin{align*}
\|P |w|^2 w\|_{L^1_{[0,t_b]}L^2_x} &\le ct_b^{1/2}\|w\|^3_{L^6_{[0,t_b]}L^6_x} \\
&\le ct_b^2 (1 + t_b^{1/6}|q|^{1/3})^3(\|\phi\|_{L^2} + t_b^{1/6}\|P\phi\|_{L^6})^9, \\
\|P |w|^2 e^{-isH_q}\phi\|_{L^1_{[0,t_b]}L^2_x} &\le ct_b^{1/2}\|w\|^2_{L^6_{[0,t_b]}L^6_x} \|e^{-isH_q}\phi\|_{L^6_{[0,t_b]}L^6_x}\\
&\le ct_b^{3/2}(1 + t_b^{1/6}|q|^{1/3})^2 (\|\phi\|_{L^2} + t_b^{1/6}\|P\phi\|_{L^6})^6\|\phi\|_{L^2}, \\
\|P w |e^{-isH_q}\phi|^2\|_{L^1_{[0,t_b]}L^2_x} &\le ct_b^{1/2}\|w\|_{L^6_{[0,t_b]}L^6_x} \|e^{-isH_q}\phi\|^2_{L^6_{[0,t_b]}L^6_x}\\
&\le ct_b(1 + t_b^{1/6}|q|^{1/3}) (\|\phi\|_{L^2} + t_b^{1/6}\|P\phi\|_{L^6})^3\|\phi\|^2_{L^2}.
\end{align*}

Putting all this together, we see that
\begin{align*}
\|h - g\|_{L^\infty_{[0,t_b]}L^2_x} \le c \bigg[&t_b^2\left(1 + t_b^{1/6}|q|^{1/3}\right)^3 (\|\phi\|_{L^2} + t_b^{1/6}\|P\phi\|_{L^6})^9 \\
&+ t_b^{3/2}\left(1 + t_b^{1/6}|q|^{1/3}\right)^2(\|\phi\|_{L^2} + t_b^{1/6}\|P\phi\|_{L^6})^6\|\phi\|_{L^2} \\
&+ t_b \left(1 + t_b^{1/6}|q|^{1/3}\right)(\|\phi\|_{L^2} + t_b^{1/6}\|P\phi\|_{L^6})^3\|\phi\|^2_{L^2}\bigg].
\end{align*}
\end{proof}

Finally we prove Lemma \ref{L:dropint}:

\begin{proof}
We write $\phi(x) = \phi_1(x) + \phi_2(x)$, where $\phi_1(x) = e^{-it_1v^2/2}e^{it_1/2}e^{ixv}\sech(x-x_0-vt_1)$, and estimate individually the eight resulting terms. We know that for large $v$, $\phi_2$ is exponentially small in $L^2$ norm from Lemma \ref{L:approx1}. This makes the term which is cubic in $\phi_1$ the largest, and we treat this one first.

I. We claim $\left\| \int_{t_1}^{t_2} e^{-i(t_2-s)H_q}|e^{-isH_q}\phi_1|^2e^{-isH_q}\phi_1 ds \right\|_{L^2_x} \le c (t_2-t_1)$.

We begin with a direct computation
\begin{align}\label{t1t2}
\indentalign \left\| \int_{t_1}^{t_2} e^{-i(t_2-s)H_q}|e^{-isH_q}\phi_1|^2e^{-isH_q}\phi_1 ds \right\|_{L^2_x} \\
&\le (t_2 - t_1) \left\| e^{-i(t_2-s)H_q}|e^{-isH_q}\phi_1|^2e^{-isH_q}\phi_1 \right\|_{L^\infty_{[t_1,t_2]}L^2_x} \\
\nonumber&\le c (t_2-t_1) \left\| |e^{-isH_q}\phi_1|^2e^{-isH_q}\phi_1 \right\|_{L^\infty_{[t_1,t_2]}L^2_x}.
\end{align}
It remains to show that this last norm is bounded by a constant. We use \eqref{eq:prop} to express $e^{-itH_q}$ in terms of $e^{-itH_0}$, recalling the formula here for the reader's convenience:
\[
e^{-itH_q}\phi_1(x) = 
\begin{aligned}[t]
&\big[e^{-itH_0}\phi_1(x) + e^{-itH_0}(\phi_1 * \rho_q)(-x)\big]x_-^0 \\
&+ e^{-itH_0}(\phi_1 * \tau_q)(x)x_+^0 + e^{\frac 12 itq^2} P \phi_1(x).
\end{aligned}
\]
This formula is only valid for functions suported in the negative half-line, but this will not cause serious difficulty and we ignore the problem for now. We first evaluate this expression with $e^{ixv} \psi(x)$ in place of $\phi_1(x)$. Here $\psi(x) = \sech(x-x_0-vt)$, and the other phase factors do not affect the norm. The first term uses the Galilean invariance of $e^{-itH_0}$ directly:
\begin{align*}
e^{-itH_0}e^{ixv}\psi(x) x_-^0 &= e^{-itv^2/2}e^{ixv}e^{-itH_0}\psi(x - vt) x_-^0.
\intertext{For the second and third terms we use in addition the fact that $e^{-itH_0}$ is a convolution operator, and convolution is associative:}
e^{-itH_0}(e^{ixv}\psi * \rho_q)(-x) x_-^0 &= \left[(e^{-itH_0}e^{ixv}\psi) * \rho_q\right](-x) x_-^0 \\
&= e^{-itv^2/2}\left[(e^{ixv}e^{-itH_0}\psi(x-vt) * \rho_q\right](-x) x_-^0 \\
&= e^{-itv^2/2}e^{-ixv}\left[(e^{-itH_0}\psi(x-vt) * (e^{-ixv}\rho_q)\right](-x) x_-^0 \\
e^{-itH_0}(e^{ixv}\psi * \tau_q)(x)x_+^0 &= e^{-itv^2/2} e^{ixv} \left[(e^{-itH_0}\psi(x-vt)) * (e^{-ixv}\tau_q)\right](x) x_+^0.
\end{align*}
The final term we leave as it is, so that we have
\begin{align}\label{phase}
e^{-itH_q}e^{ivx}\psi(x) = & e^{-itv^2/2}e^{ixv}\big[e^{-itH_0}\psi(x - vt) \\
&+ \left[(e^{-itH_0}\psi(x-vt) * (e^{-ixv}\rho_q)\right](-x)\big]x_-^0 \\
\nonumber&+ e^{-itv^2/2} e^{ixv} \left[(e^{-itH_0}\psi(x-vt)) * (e^{-ixv}\tau_q)\right](x) x_+^0 \\
&+ e^{\frac 12 itq^2} P e^{ivx}\psi(x).
\end{align}
Before proceeding to the estimate of (\ref{t1t2}) we introduce the following notation:
\[
f_-(x) = f(x) x^0_-, \qquad f_+(x) = f(-x) x^0_-
\]
so that $f = Rf_+ + f_-$, where $Rf(x) = f(-x)$, and (\ref{phase}) will be applicable to $f_\pm$. We then write
\[
|e^{-isH_q}\phi_1|^2e^{-isH_q}\phi_1 = R(|e^{-isH_q}\phi_{1+}|^2e^{-isH_q}\phi_{1+}) + \cdots + |e^{-isH_q}\phi_{1-}|^2e^{-isH_q}\phi_{1-},
\]
with a total of 8 terms. After applying (\ref{phase}) three times to each of them, we will have $8 \cdot 4^3$ terms. We will estimate these terms in groups. We observe first that the distinction between $\phi_{1+}$ and $\phi_{1-}$ will not play a role, and neither will the presence or absence of $R$. We accordingly write $\psi$ for $\sech(x-x_0-vt_1)_\pm$ and omit $R$ when it appears. In what follows $\sigma_q$ denotes either $e^{-ixv}\rho_q$, $e^{-ixv}\tau_q$, or $\delta_0$, each of which has $L^1$ norm $1$.
\begin{align*}
\indentalign \left\| \left|(e^{-isH_0}\psi(x-vs)) * \sigma_q\right|^2(e^{-isH_0}\psi(x-vs))*\sigma_q \right\|_{L^\infty_{[t_1,t_2]}L^2_x} \\
&\le c \left\|\left|(e^{-isH_0}\psi(x-vs)) * \sigma_q\right|^2(e^{-isH_0}\psi(x-vs))*\sigma_q \right\|_{L^\infty_{[t_1,t_2]}L^2_x}
\end{align*}
Now we pass from $L^2$ to $L^\infty$ for two of the factors, and use Young's inequality:$\|f * g\|_p \le \|f\|_p\|g\|_1$, once for $p=2$ and twice for $p = \infty$, and then use the Gagliardo-Nirenberg-Sobolev inequality which states that the $L^\infty$ norm is controlled by the $H^1$ norm. Because the $H^1$ norm is preserved by $e^{-isH_0}$, we are home free:
\begin{align*}
&\le c \left\|(e^{-isH_0}\psi(x-vs))*\sigma_q \right\|_{L^\infty_{[t_1,t_2]}L^2_x}\left\|(e^{-isH_0}\psi(x-vs))*\sigma_q \right\|_{L^\infty_{[t_1,t_2]}L^\infty_x}^2 \\
&\le c \|\sech(x)\|_{L^\infty_{[t_1,t_2]}L^2_x}\|\sech(x)\|^2_{L^\infty_{[t_1,t_2]}H^1_x} \le c.
\end{align*}
Terms where one or more of the factors of $(e^{-isH_0}\psi(x-vs))$ are replaced by $e^{\frac 12 itq^2} P e^{ivx}\psi(x)$ are treated in the same way. We have
\[
\left\|e^{\frac 12 itq^2} P e^{ivx}\psi(x)\right\|_{L^\infty_{[t_1,t_2]}L^p_x} \le \|\psi(x)\|_{L^\infty_{[t_1,t_2]}L^p_x},
\]
where $p$ is either $2$ or $\infty$.

II. For the other terms the phases will play no role, because we will use Strichartz estimates. The smallness will come more from the smallness of $\phi_2$ than from the brevity of the time interval.
$$\bigg\| \int_{t_1}^{t_2} e^{-i(t_2-s)H_q}|e^{-isH_q}\phi_1|^2e^{-isH_q}\phi_2 ds \bigg\|_{L^2_x} \le c \||e^{-isH_q}\phi_1|^2e^{-isH_q}\phi_2\|_{L^1_{[t_1,t_2]}L^2_x}.$$
We have used the Strichartz estimate with $(\tilde p, \tilde r) = (1,2)$, and combined the resulting terms using \eqref{E:PHolder}. We use H\"older's inequality so as to put ourselves in a position to reapply the Strichartz estimate.
\begin{align*}
&\le c (t_2-t_1)^{\frac 12}\|e^{-isH_q}\phi_1\|^2_{L^4_{[t_1,t_2]}L^\infty_x}\|e^{-isH_q}\phi_2\|_{L^\infty_{[t_1,t_2]}L^2_x} \\
&\le c (t_2-t_1)^{\frac 12} \left(\|\phi_1\|_{L^2_x} + (t_2-t_1)^{\frac 1 4}\|P\phi_1\|_{L^\infty_x}\right)^2\left(\|\phi_2\|_{L^2_x} + \|P\phi_2\|_{L^2_x}\right).
\intertext{Once again we use \eqref{E:PHolder} to combine terms, this time with a penalty in $|q|$.}
&\le c (t_2-t_1)^{\frac 12} \la q\ra \|\phi_1\|_{L^2_x}^2\|\phi_2\|_{L^2_x} \\
&\le c (t_2-t_1)^{\frac 12} \la q \ra^\frac 32 |q|^\frac 32 e^{-v^\eps}.
\end{align*}
Similarly we find that
\begin{align*}
\bigg\| \int_{t_1}^{t_2} e^{-i(t_2-s)H_q}|e^{-isH_q}\phi_2|^2e^{-isH_q}\phi_1 ds \bigg\|_{L^2_x} &\le c (t_2-t_1)^{\frac 12}\la q \ra^2 |q|^3 e^{-2v^\eps} \\
\bigg\| \int_{t_1}^{t_2} e^{-i(t_2-s)H_q}|e^{-isH_q}\phi_2|^2e^{-isH_q}\phi_2 ds \bigg\|_{L^2_x} &\le c (t_2-t_1)^{\frac 12}\la q \ra^\frac 52 |q|^\frac 92 e^{-3v^\eps}.
\end{align*}
\end{proof}

\section{Phase 3}
\label{S:phase3}

Let $t_3=t_2+\eps\log v$.  Label
$$u_\trans(x,t) =e^{-itv^2/2}e^{it_2/2}e^{ixv}\nlso(t-t_2)[t(v)\sech(x)](x-x_0-tv)$$
for the transmitted (right-traveling) component and
$$u_\refl(x,t) = e^{-itv^2/2}e^{it_2/2}e^{-ixv}\nlso(t-t_2)[r(v)\sech(x)](x+x_0+tv)$$
for the reflected (left-traveling) component.  By Appendix A from \cite{HMZ}, for each $k\in \mathbb{N}$ there exists a constant $c_k>0$ and an exponent $\sigma(k)>0$ such that
\begin{equation}
\label{E:trans}
\|u_\trans(x,t)\|_{L_{x<0}^2} \leq \frac{c_k(\log v)^{\sigma(k)}}{v^{k\eps}}, \qquad  |u_\trans(0,t)| \leq \frac{c_k(\log v)^{\sigma(k)}}{v^{k\eps}}
\end{equation}
and
\begin{equation}
\label{E:refl}
\|u_\refl(x,t)\|_{L_{x>0}^2} \leq \frac{c_k(\log v)^{\sigma(k)}}{v^{k\eps}}, \qquad  |u_\refl(0,t)| \leq \frac{c_k(\log v)^{\sigma(k)}}{v^{k\eps}}
\end{equation}
both uniformly on the time interval $[t_2,t_3]$.  

Let us first give an outline of the argument in this section.  We would like to control $\|w\|_{L_{[t_2,t_3]}^\infty L_x^2}$, where $w = u - u_{\tr} - u_{\refl}$.  If, after subdividing the interval $[t_2,t_3]$ into unit-sized intervals, we could argue that $w(t)$ at most doubles (or more accurately, multiplies by a fixed constant independent of $q$ and $v$) over each interval, then we could take $t_3\sim \log v$ and conclude that the size of $\|w(t)\|_{L_x^2}$ would only increment by at most a small positive power of $v$ over $[t_2,t_3]$.  This was the strategy employed in \cite{HMZ}.  The equation for $w$ induced by the equations for $u$, $u_{\tr}$, and $u_{\refl}$ took the form
\begin{equation}
\label{E:w10}
0=i\partial_t w + \tfrac12\partial_x^2 w - q\delta_0 w + F
\end{equation}
where $F$ involved product terms of the form (omitting complex conjugates) $w^3$, $w^2u_{\tr}$, $w^2u_{\refl}$, $w u_{\tr}^2$, $w u_{\tr} u_{\refl}$, and $w u_{\refl}^2$, $u_{\refl}^2u_{\tr}$, $u_{\tr}^2u_{\refl}$, $q\delta_0 u_{\tr}$, $q\delta_0 u_{\refl}$.  The Strichartz estimates were applied to \eqref{E:w10} to deduce a bound on $\|w\|_{L_{[t_a,t_b]}^p L_x^r}$ for all admissible pairs $(p,r)$ over unit-sized time intervals $[t_a,t_b]$.  Although a bound for $(p,r)=(\infty,2)$ would have sufficed, the analysis forced the use of the full range of admissible pairs $(p,r)$ since such norms necessarily arose on the right-hand side of the estimates.

A direct implementation of this strategy does not work for $q<0$.  The difficulty stems from the fact that the initial data $w(t_a)$ for the time interval $[t_a,t_b]$ has a nonzero projection onto the eigenstate $|q|^{1/2}e^{-|q||x|}$.  While the perturbed linear flow of this component is adequately controlled in $L_{[t_a,t_b]}^\infty L_x^2$, it is equal to a positive power of $|q|$ when evaluated in other Strichartz norms.  For example, $\| e^{\frac12itq^2}|q|^{1/2}e^{-|q||x|} \|_{L_{[t_a,t_b]}^6L_x^6} = |q|^{1/3}$.  Thus each iterate over a unit-sized time interval $[t_a,t_b]$ will result in a multiple of $|q|^{1/3}$, and we cannot carry out more than two iterations.

Our remedy is to separate from the above $w = u - u_{\tr} - u_{\refl}$ at $t=t_a$ the projection onto the eigenstate $|q|^{1/2}e^{-|q||x|}$, and evolve this piece by the $\nlsq$ flow.  Specifically, we set
$$u_\bd(t) = \nlsq(t)\big[ P\big(u(t_a)-u_{\tr}(t_a)-u_{\refl}(t_a)\big) \big]$$
and then model $u(t)$ as
$$u(t) = u_{\tr}(t)+u_{\refl}(t)+u_{\bd}(t)+w(t)$$
This equation \emph{redefines} $w(t)$ from that discussed above, and it now has the property that $w(t_a)$ is orthogonal to the eigenstate $|q|^{1/2}e^{-|q||x|}$.  We will, over the interval $[t_a,t_b]$, estimate $w(t)$ in the full family of Strichartz norms but will only put the norms $L_{[t_a,t_b]}^\infty L_x^2$ and $L_{[t_a,t_b]}^\infty \dot H_x^1$ on $u_{\bd}(t)$.  These norms are controlled by \emph{nonlinear} information: the $L^2$ conservation and energy conservation of the $\nlsq$ flow.  The use here of nonlinear information is the key new ingredient; perturbative linear estimates for $u_{\bd}$ are too weak to complete the argument.

The equation for $w(t)$ induced by the equations for $u(t)$, $u_{\bd}(t)$, $u_{\tr}(t)$, and $u_{\refl}(t)$ takes the form
$$0=i\partial_t w + \tfrac12\partial_x^2 w -q\delta_0(x)w + F$$
where $F$ contains terms of the following types (ignoring complex conjugates):
\begin{itemize}
\item (delta terms) $u_{\tr}\delta_0$, $u_{\refl}\delta_0$ 
\item (cubic in $w$) $w^3$ 
\item (quadratic in $w$) $w^2u_{\tr}$, $w^2u_{\refl}$, $w^2u_{\bd}$ 
\item (linear in $w$) $wu_{\tr}^2$, $wu_{\refl}^2$, $wu_{\bd}^2$, $wu_{\refl}u_{\tr}$, $w u_{\bd}u_{\tr}$, $wu_{\bd}u_{\refl}$
\item (interaction) $u_{\bd}u_{\tr}^2$, $u_{\bd}u_{\refl}^2$, $u_{\refl}u_{\bd}^2$, $u_{\refl}u_{\tr}^2$, $u_{\tr}u_{\bd}^2$, $u_{\tr}u_{\refl}^2$ 
\end{itemize}
The integral equation form of $w$ is
$$w(t) = 
\begin{aligned}[t]
& U_q(t)[(1-P)(u(t_a)-u_{\tr}(t_a)-u_{\refl}(t_a))] \\
&-i \int_0^t \Big( e^{\frac12i(t-t')q^2}P F(t') + U_q(t-t')(1-P)F(t') \Big) \, dt'
\end{aligned}
$$
We estimate $w$ in the full family of Strichartz norms, and encounter the most adverse powers of $|q|$ in the $PF$ component of the Duhamel term.  Of the terms making up $F$, the most difficult are the ``interaction'' terms listed above that involve at least one $u_{\bd}$.   Let $A \defeq\|u(t_a) - u_{\tr}(t_a) - u_{\refl}(t_a)\|_{L_x^2}$.  Then we find from the $L^2$ conservation and energy conservation (see Lemma \ref{L:energycon}) of the $\nlsq$ flow, that
$$\|u_{\bd}\|_{L_{[t_a,t_b]}^\infty L_x^2} \leq A, \quad \|\partial_x u_{\bd}\|_{L_{[t_a,t_b]}^\infty L_x^2} \leq c |q|A\, .$$
Using these bounds, we are able to control the interaction terms $u_{\tr}u_{\bd}$ and $u_{\refl}u_{\bd}$ as
$$\|u_{\tr}u_{\bd}\|_{L_{[t_a,t_b]}^\infty L_x^2} + \|u_{\refl}u_{\bd}\|_{L_{[t_a,t_b]}^\infty L_x^2} \leq |q|^{-\frac12}A+|q|^1A^3$$
(see Lemma \ref{L:interaction}).  The estimate of the interaction terms comes with a factor $|q|^{1/2}$, and thus the bound on the interaction terms with at least one copy of $u_{\bd}$ is of size $A+|q|^\frac32A^3$.  We want this to be at most comparable to $A$, so we need $|q|^\frac32A^2 \lesssim 1$.  But the estimates of Phase 2 leave us starting Phase 3 with an error of size $|q|^\frac13 v^{-\frac76(1-\eps)}$.  Let us assume, as a bootstrap assumption, that we are able to maintain a control of size $|q|^\frac13 v^{-\frac76(1-2\eps)}$ on the error.   Then, even in the worst case in which $A\sim |q|^\frac13 v^{-\frac76(1-2\eps)}$, the condition $|q|^\frac32A^2 \lesssim 1$ is implied by $v \gtrsim |q|^{\frac{13}{14}(1+2\eps)}$.  So, if we impose the assumption $v \gtrsim |q|^{\frac{13}{14}(1+2\eps)}$, we can carry out the iterations, with the error at most doubling over each iterate.  If it begins at $|q|^\frac13 v^{-\frac76(1-\eps)}$, then after $\sim \eps\log v$ iterations, it is no more than $|q|^\frac13 v^{-\frac76(1-2\eps)}$, the size of the bootstrap assumption.

Now, to prove the error bound of size $|q|^\frac13 v^{-\frac76(1-\eps)}$ in the Phase 2 analysis required the introduction of the cubic correction refinement; the shorter argument in \cite{HMZ} would only have provided a bound of size $|q|^\frac16v^{-\frac12+\eps}$.   The bound $|q|^\frac16v^{-\frac12+\eps}$ combined with the condition  $|q|^\frac32A^2 \lesssim 1$  gives the requirement $v\gtrsim |q|^{\frac{11}{6}+\eps}$, which is unacceptable since the most interesting phenomena occur for $|q|\sim v$.  This is the reason we \emph{needed} the cubic correction refinement in Phase 2.

We now give the full details of the argument.  We begin with an energy estimate for our differential equation:

\begin{lem}
\label{L:energycon}
If $u$ satisfies
\[ i\D_t u + \frac 1 2 \D_x^2 u - q \delta_0 u + |u|^2u = 0, \]
then
\begin{equation}\label{E:energy}
\|\D_x u(x,t)\|_{L^2_x}  \le 2\|\D_x u(x,0)\|_{L^2_x} + 2|q| \|u(x,0)\|_{L^2_x} + \|u(x,0)\|^3_{L^2_x}.
\end{equation}
\end{lem}

\begin{proof}
Multiplying by $\D_t\overline{u}$, intergrating in space, and taking the real part, we see that
\begin{align*}
\textrm{Re} \left(\frac 1 2 \int \D_x^2 u \D_t \overline{u} dx - q u(0,t)\D_t\overline{u}(0,t) + \int |u|^2u \D_t\overline{u} dx\right) &= 0.
\intertext{Here we multiply by 4 and integrate by parts in the first term:}
-\D_t\int |\D_x u|^2 dx - 2q\D_t|u(0,t)|^2 + \D_t\int |u|^4 dx &= 0
\end{align*}
Integrating from $0$ to $t$ and solving for $\|\D_x u(t)\|^2_{L^2_x}$, we find that
$$
\|\D_x u(x,t)\|^2_{L^2_x} = 
\begin{aligned}[t]
&\|\D_x u(x,0)\|^2_{L^2_x} + 2 q |u(0,0)|^2 - 2q |u(0,t)|^2 \\
&+ \|u(x,t)\|^4_{L^4_x} - \|u(x,0)\|^4_{L^4_x}.
\end{aligned}
$$
Dropping the terms with a favorable sign from the right hand side, we see that
$$\|\D_x u(x,t)\|^2_{L^2_x} \le \|\D_x u(x,0)\|^2_{L^2_x} + 2|q| |u(0,t)|^2 + \|u(x,t)\|^4_{L^4_x}$$
Next, using $\|u\|^2_{L^4_x} \le \|u\|_{L^2_x}\|u\|_{L^\infty_x}$, together with $\|u\|^2_{L^\infty_x} \le \|u\|_{L^2_x}\|\D_xu\|_{L^2_x}$, we have
$$\|\D_x u(x,t)\|^2_{L^2_x} \le 
\begin{aligned}[t]
&\|\D_x u(x,0)\|^2_{L^2_x} + 2|q| \|u(x,t)\|_{L^2_x}\|\D_xu(x,t)\|_{L^2_x} \\
&+ \|u(x,t)\|^3_{L^2_x}\|\D_x u(x,t)\|_{L^2_x}
\end{aligned}
$$
Here we use $\|u(x,t)\|_{L^2_x} = \|u(x,0)\|_{L^2_x}$:
$$\|\D_x u(x,t)\|^2_{L^2_x}  \le \|\D_x u(x,0)\|^2_{L^2_x} + \left(2|q| \|u(x,0)\|_{L^2_x} + \|u(x,0)\|^3_{L^2_x} \right)\|\D_xu(x,t)\|_{L^2_x}$$
Using $ab \le \frac 12 a^2 + \frac 1 2 b^2$ to solve for $\|\D_x u(x,t)\|$, we obtain
$$\|\D_x u(x,t)\|^2_{L^2_x}  \le 2\|\D_x u(x,0)\|^2_{L^2_x} + \left(2|q| \|u(x,0)\|_{L^2_x} + \|u(x,0)\|^3_{L^2_x} \right)^2.$$
This implies the desired result.
\end{proof}

We now give an approximation lemma analogous to that in Phase 1, but with the error divided into two parts.

\begin{lem}\label{L:approx6}
Suppose $t_a < t_b$ and $t_b - t_a \le c_1$. Let $u_\bd(t)$ be the flow of $\nlsq$ with initial condition
\begin{equation}\label{ubd}
u_{\bd}(t_a) = P[u(t_a) - u_\trans(t_a) - u_\refl(t_a)],
\end{equation}
and put 
\begin{equation}
\label{E:wdef}
w = u - u_\trans - u_\refl - u_\bd \,.
\end{equation}
Suppose $|q|^{\frac 12} \|u_{\bd}(t_a)\|_{L^2_x} \le 1$ and, for some $k \in \NN$,
\[\|w(t_a)\|_{L^2_x} + c(k)\la q\ra^2\frac{(\log v)^{\sigma(k)}}{v^{k\eps}} + \|u_\bd(t_a)\|_{L^2_x} + |q| \la q \ra^\frac 12 \|u_{\bd}(t_a)\|^3_{L^2_x} \le c_2 \, .\]
Then 
\[
\|w\|_{L^\infty_{[t_a,t_b]}L^2_x} \le c_3 
\begin{aligned}[t]
\bigg( &\|w(t_a)\|_{L^2_x} + c(k)\la q\ra^2\frac{(\log v)^{\sigma(k)}}{v^{k\eps}} \\
&+ \|u_\bd(t_a)\|_{L^2_x} + |q| \la q \ra^\frac 12 \|u_{\bd}(t_a)\|^3_{L^2_x}\bigg).\end{aligned}
\]
The constants $c_1$, $c_2$, and $c_3$ are independent of $q$, $v$, and $\eps$.
\end{lem}

Once again we apply our lemma before proving it. Suppose now that for some large $c$ and $k$, where $c$ is absolute and $k$ depends on $\eps$, we have $|q|\la q\ra^\frac 12 \|w(t_2)\|^2\le c$ and $c(k)\la q\ra^2\frac{(\log v)^{\sigma(k)}}{v^{k\eps}}\le c \|w(t_2)\|_{L^2_x}$. These additional assumptions cause the conclusion of the lemma to become $\|w\|_{L^\infty_{[t_a,t_b]}L^2_x} \le c\|w(t_a)\|_{L^2_x}$, and they will follow from assuming $|q|^{10}\la q \ra^3 v^{-14(1-2\eps)} \le c$ and $\la q \ra v^{-n} \le c_{\eps,n}$, where $c$ is an absolute constant and $c_{\eps,n}$ is a small constant, dependent only on $\eps$ and $n$, which goes to zero when $\eps \to 0$ or when $n \to \infty$. Let $m$ be the integer such that $mc_1< \eps \log v < (m+1)c_1$.  We apply Lemma \ref{L:approx6} successively on the intervals $[t_2,t_2+c_1], \ldots, [t_2+(m-1)c_1,t_2+mc_1]$ as follows. On $[t_2,t_2+c_1]$, we obtain
\[ \|w(\cdot,t)\|_{L_{[t_2,t_2+1]}^\infty L_x^2} \leq c_3\|w(\cdot,t_2)\|_{L_x^2}. \]
Applying Lemma \ref{L:approx6} on $[t_2+c_1,t_2+2c_1]$ and combining with the above estimate,
\[\|w(\cdot,t)\|_{L_{[t_2+1,t_2+2]}^\infty L_x^2} \leq c_3^2\|w(\cdot,t_2)\|_{L_x^2}.\]
Continuing up to the $k$-th step and then collecting all of the above estimates,
\[\|w(\cdot,t)\|_{L_{[t_2,t_3]}^\infty L_x^2} \leq c_3^{m+1}\|w(\cdot,t_2)\|_{L_x^2} \le c v^\eps \|w(\cdot,t_2)\|_{L_x^2} \le c\left(v^{-1+2\eps} +  |q|^{\frac 13} v^{-\frac 76 (1-2\eps)} \right). \]

Notice that $u_{\bd}(t)$ and $w(t)$ are \emph{redefined} by \eqref{ubd} and \eqref{E:wdef} as we move from one interval $I_j\defeq [t_2+(j-1)c_1,t_2+jc_1]$ to the next interval $I_{j+1}\defeq [t_2+jc_1,t_2+(j+1)c_1]$ in the above iteration argument.  That is, on $I_j$, $u_{\bd}(t)$ is the $\nlsq$ flow of an initial condition at $t=t_2+(j-1)c_1$, and on $I_{j+1}$, $u_{\bd}(t)$ is the $\nlsq$ flow of an initial condition at $t=t_2+jc_1$ that does not necessarily match the value of the previous flow at that point.  In other words, at the interface of these two intervals,
$$\lim_{t\nearrow t_2+jc_1} u_\bd(t) \neq \lim_{t\searrow t_2+jc_1} u_\bd(t)$$

It remains to prove Lemma \ref{L:approx6}. On the way we will need to estimate the overlap of $u_\bd$ with $u_\trans$ and $u_\refl$.

\begin{lem}
\label{L:interaction}
Let $t_a < t_b$, let $u_\trans$ and $u_\refl$ satisfy \eqref{E:trans} and \eqref{E:refl}, and let $u_\bd$ be the flow under $\nlsq$ of an initial condition proportional to $e^{q|x|}$. Then
\begin{equation}\label{E:bdtr}
\begin{aligned}
\indentalign \|u_\trans u_{\bd}\|_{L^\infty_{[t_a,t_b]}L^2_x} + \|u_{\refl}u_{\bd}\|_{L^\infty_{[t_a,t_b]}L^2_x} \\
&\le c \left(|q|^{-\frac 12}\|u_\bd(t_a)\|_{L^2_x} + |q|^\frac 12 \la q \ra^\frac 12 \|u_{\bd}(t_a)\|^3_{L^2_x} \right).
\end{aligned}
\end{equation}
\end{lem}

\begin{proof}
We prove the inequality only for $\|u_\trans u_{\bd}\|_{L^\infty_{[t_a,t_b]}L^2_x}$, the argument for $\|u_{\refl}u_{\bd}\|_{L^\infty_{[t_a,t_b]}L^2_x}$ being identical. We will use the following formula for $u_{\bd}$:
\[
u_{\bd}(x,t) = c_u e^{\frac i 2 t v^2} |q|^{\frac 1 2} e^{q|x|} + \int_{t_a}^t e^{-i(t-s)H_q} |u_{\bd}(x,s)|^2 u_{\bd}(x,s) ds,
\]
where $c_u$ is the overlap of $u_\bd$ with the linear eigenstate, and is a constant bounded by the $L^2_x$ norm of $u_\bd$. We then have
\[
\|u_\trans u_{\bd}\|_{L^\infty_{[t_a,t_b]}L^2_x} \le 
\begin{aligned}[t]
& c_u|q|^{\frac 1 2}\left\|u_\trans e^{q|x|}\right\|_{L^\infty_{[t_a,t_b]}L^2_x} \\
&+ \left\|u_\trans \int_{t_2}^t e^{-i(t-s)H_q} |u_{\bd}(s)|^2 u_{\bd}(s) ds \right\|_{L^\infty_{[t_a,t_b]}L^2_x}
\end{aligned}
\]

For the first term it is enough that the linear eigenstate has $L^1$ norm proportional to $|q|^{-\frac 12}$. We remark that a better estimate is possible using the fact that $u_\bd$ and $u_\trans$ are concentrated in different parts of the real line.
\begin{align*}
\int |u_\trans(x,t)|^2e^{2q|x|} dx &\le c\|u_\trans(x,t)\|^2_{L^\infty_x} |q|^{-1} \le c |q|^{-1}
\end{align*}
Putting back in the factor of $\|u_\bd(t_a)\|_{L^2_x}|q|^\frac 12$, we obtain the first term of the desired estimate.

Now we treat the integral term.
\begin{align*}
\indentalign \left\|u_\trans\int_{t_2}^t e^{-i(t-s)H_q} |u_{\bd}(s)|^2 u_{\bd}(s) ds \right\|_{L^\infty_{[t_a,t_b]}L^2_x} \\
&\le c\left\|\int_{t_2}^t e^{-i(t-s)H_q} |u_{\bd}(s)|^2 u_{\bd}(s) ds \right\|_{L^\infty_{[t_a,t_b]}L^2_x},
\intertext{where we have used the explicit formula for $u_\trans$ to take its supremum. We now use a Strichartz estimate:}
&\le c \left(\left\||u_{\bd}|^3\right\|_{L^{4/3}_{[t_a,t_b]}L^1_x} +  \left\|P|u_{\bd}|^3\right\|_{L^1_{[t_a,t_b]}L^2_x}\right) \\
&\le c |q|^{\frac 12} \|u_{\bd}^3\|_{L^\infty_{[t_a,t_b]}L^1_x} \\
&\le c |q|^{\frac 12} \|u_{\bd}\|^2_{L^\infty_{[t_a,t_b]}L^2_x}\|u_{\bd}\|_{L^\infty_{[t_a,t_b]}L^\infty_x}.
\intertext{Here we use $\|f\|^2_{L^\infty} \le \|f\|_{L^2_x}\|\D f\|_{L^2_x}$.}
&\le c |q|^{\frac 12} \|u_{\bd}\|^{\frac 52}_{L^\infty_{[t_a,t_b]}L^2_x}\|\D_x u_{\bd}\|^{\frac 12}_{L^\infty_{[t_a,t_b]}L^2_x}
\end{align*}
Applying \eqref{E:energy} gives
\[\le c |q|^{\frac 12} \|u_{\bd}\|^{\frac 52}_{L^\infty_{[t_a,t_b]}L^2_x} \left(\|\D_x u_{\bd}(x,t_a)\|^{\frac 12}_{L^2_x} + |q|^\frac 12 \|u_{\bd}(x,t_a)\|^\frac 12_{L^2_x} + \|u_{\bd}(x,t_a)\|^\frac 32_{L^2_x}\right).\]
But the $L^2_x$ norm of $u_{\bd}$ is conserved, so that $\|u_{\bd}\|_{L^\infty_{[t_a,t_b]}L^2_x} = \|u_\bd(t_a)\|_{L^2_x}$. Moreover, because $u_\bd(t_a)$ is proportional to $e^{q|x|}$, we have $|\D_x u_\bd(x,t_a)| = |q| |u_\bd(x,t_a)|$, giving $\|\D_x u_\bd(x,t_a)\|_{L^2_x} = |q|\|u_\bd(t_a)\|_{L^2_x}$. This allows us to conclude
\[ = c |q|^{\frac 12} \|u_{\bd}(t_a)\|^{\frac 52}_{L^2_x} \left(|q|^{\frac 12}\|u_\bd(t_a)\|^{\frac 12}_{L^2_x} + |q|^\frac 12 \|u_{\bd}(x,t_a)\|^\frac 12_{L^2_x} + \|u_{\bd}(x,t_a)\|^\frac 32_{L^2_x}\right)\]
Combining terms, we obtain
\[ \le c |q|^\frac 12 \la q\ra^\frac 12 \|u_{\bd}(t_a)\|^3_{L^2_x}, \]
giving us the second term of the conclusion.
\end{proof}

Now we proceed to the proof of Lemma \ref{L:approx6}.

\begin{proof}
We begin by computing the differential equation solved by $w$.
\begin{align*}
0 &= i \D_t u + \frac 1 2 \D_x^2u - q\delta_0(x) u + u|u|^2 \\
&= \begin{aligned}[t]
&i \D_t w + \frac 1 2 \D_x^2w - q\delta_0(x) u_{\trans} - q\delta_0(x) u_{\refl} \\
&+ (u_{\trans} + u_{\refl} + u_{\bd} + w)|u_{\trans} + u_{\refl} + u_{\bd} + w|^2\\
& - |u_{\trans}|^2u_{\trans} - |u_{\refl}|^2u_{\refl} - |u_{\bd}|^2u_{\bd}.
\end{aligned}
\end{align*}
This can also be written as the following integral equation:
\begin{equation}
\label{E:intw}
\begin{aligned}
w(t) = e^{-i(t-t_2)H_q}w(t_2)+ \int_{t_2}^t &e^{-i(t-s)H_q} \big[ q\delta_0(x) u_{\trans}(s) + q\delta_0(x) u_{\refl} \\
&- |u_{\trans}(s)|^2u_{\refl}(s) - \cdots - |w(s)|^2w(s)\big] ds. 
\end{aligned}
\end{equation}

For this estimate we will use $\|w\|_X = \|w\|_{L^\infty_{[t_a,t_b]}L^2_x} + \|w\|_{L^4_{[t_a,t_b]}L^\infty_x}$, and accordingly estimate the $L^p_{[t_a,t_b]}L^r_x$ norms of each of the terms on the right hand side of \eqref{E:intw} for $(p,r)$ equal to either $(\infty,2)$ or $(4,\infty)$. We choose these norms in order to be able to apply the Strichartz estimate.

First we treat the term arising from the initial condition. Here we use the fact that by construction, our initial condition satisfies $Pw_2(t_a) = 0$.
\begin{align*}
\left\| e^{-i(t-t_a)H_q}w(t_a) \right \|_{L^p_{[t_a,t_b]}L^r_x} \le c \|w(t_a)\|_{L^2_x}
\end{align*}

Second we treat the delta terms, for which we obtain the following:
\begin{align*}
\left\| \int_{t_a}^t e^{-i(t-s)H_q} q\delta_0(x) u_{\trans}(s)ds \right\|_{L^p_{[t_a,t_b]}L^r_x} &\le c \la q\ra^2\|u_{\trans}(0,t)\|_{L^\infty_{[t_a,t_b]}},\\
\left\| \int_{t_a}^t e^{-i(t-s)H_q} q\delta_0(x) u_{\refl}(s)ds \right\|_{L^p_{[t_a,t_b]}L^r_x} &\le c \la q\ra^2\|u_{\refl}(0,t)\|_{L^\infty_{[t_a,t_b]}}.
\end{align*}
We have used here the bound on $t_b - t_a$ to pass from the $L^{4/3}_{[t_a,t_b]}$ norm to the $L^\infty_{[t_a,t_b]}$ norm, and we have combined all the terms under the one with the least favorable power of $\la q \ra$. These last expressions are estimated using \eqref{E:trans} and \eqref{E:refl} respectively, to obtain, in both cases
$$\le c \la q\ra^2\frac{c(k)(\log v)^{\sigma(k)}}{v^{k\eps}}.$$

Third we the treat term which is cubic in $w$. For this we have
\begin{align*}
\left\| \int_{t_a}^t e^{-i(t-s)H_q} |w(s)|^2w(s) ds \right\|_{L^p_{[t_a,t_b]}L^r_x} &\le c \left(\|w^3\|_{L^{6/5}_{[t_a,t_b]}L^{6/5}_x} + \|Pw^3\|_{L^1_{[t_a,t_b]}L^r_x} \right) \\
&\le c \left(\|w\|^3_{L^{18/5}_{[t_a,t_b]}L^{18/5}_x} + \|w\|^3_{L^3_{[t_a,t_b]}L^{3r}_x} \right).
\intertext{At this stage we use $\|w\|_{L^{18/5}_x} \le \|w\|^{5/9}_{L^2_x}\|w\|^{4/9}_{L^\infty_x}$, and an analogous inequality for $\|w\|_{L^{3r}_x}$. We then use the boundedness of $t_b - t_a$ to pass to the $L^\infty_{[t_a,t_b]}L^2_x$ and $L^4_{[t_a,t_b]}L^\infty_x$ norms, giving us}
&\le c \|w\|^3_X.
\end{align*}

Fourth we treat terms of the form $w^2u_{\bd}$, as usual ignoring complex conjugates. This time we use $\tilde p = 1$, $\tilde r = 2$ on the right hand side of the Strichartz estimate.
\begin{align*}
\indentalign \left\| \int_{t_a}^t e^{-i(t-s)H_q} w(s)^2 u_{\bd}(s) ds \right\|_{L^p_{[t_a,t_b]}L^r_x} \\
&\le c \left(\|w^2 u_{\bd}\|_{L^1_{[t_a,t_b]}L^2_x} + \|Pw^2u_{\bd}\|_{L^1_{[t_a,t_b]}L^r_x} \right) \\
&\le c \left(\|w^2 u_{\bd}\|_{L^1_{[t_a,t_b]}L^2_x} + |q|^{\frac 12}\|w^2u_{\bd}\|_{L^1_{[t_a,t_b]}L^2_x} \right)
\intertext{We have used $\|Pf\|_{L^{r_1}} \le c|q|^{\frac 1 {r_2} -\frac 1 {r_1}}\|f\|_{L^{r_2}}$ to pass to the $L^2_x$ norm, incurring a penalty no worse than $|q|^{1/2}$. This allows us to put an $L^2_x$ norm on $u_{\bd}$, which is our preferred norm for this part of the error.}
&\le c |q|^{\frac 12} \|u_{\bd}\|_{L^\infty_{[t_a,t_b]}L^2_x} \|w\|^2_X.
\intertext{Here we use $|q|^{\frac 12} \|u_{\bd}\|_{L^\infty_{[t_a,t_b]}L^2_x} \le 1$:}
&\le c \|w\|^2_X.
\end{align*}

Fifth we treat terms of the form $w^2u_{\trans}$ and $w^2u_{\refl}$. These terms obey the same estimate, and we write out the computation in one case only.
\begin{align*}
\indentalign \left\| \int_{t_a}^t e^{-i(t-s)H_q} w(s)^2 u_{\trans}(s) ds \right\|_{L^p_{[t_a,t_b]}L^r_x} \\
&\le c \left(\|w^2 u_{\trans}\|_{L^1_{[t_a,t_b]}L^2_x} + \|Pw^2u_{\trans}\|_{L^1_{[t_a,t_b]}L^r_x} \right).
\intertext{The first term is the same as in the $u_{\bd}$ case. For the second term we pass to the $L^\infty_x$ norm, so as not to be penalized in $|q|$.}
&\le c \left(\|w\|^2_X \|u_{\trans}\|_{L^\infty_{[t_a,t_b]}L^2_x} + \|w^2u_{\trans}\|_{L^1_{[t_a,t_b]}L^\infty_x} \right)
\intertext{Here we use the fact that our explicit formula for $u_{\trans}$ allows us to control all of its mixed norms.}
&\le c \|w\|^2_X.
\end{align*}

Sixth we treat terms of the form $wu_{\trans}^2$, $wu_{\refl}^2$, and $wu_{\trans}u_{\refl}$, again writing out the computation in one case only.
\begin{align*}
\indentalign \bigg\| \int_{t_a}^t e^{-i(t-s)H_q} w(s) u_{\trans}^2(s) ds \bigg\|_{L^p_{[t_a,t_b]}L^r_x} \\
&\le c \left(\|w u_{\trans}^2\|_{L^1_{[t_a,t_b]}L^2_x} + \|Pwu_{\trans}^2\|_{L^1_{[t_a,t_b]}L^r_x} \right) \\
&\le c \left(\|w\|_{L^1_{[t_a,t_b]}L^2_x}\|u_{\trans}\|^2_{L^\infty_{[t_a,t_b]}L^\infty_x} + \|w\|_{L^1_{[t_a,t_b]}L^\infty_x} \|u_{\trans}\|^2_{L^\infty_{[t_a,t_b]}L^\infty_x} \right) \\
&\le c (t_a-t_b)^{3/4}\left(\|w\|_{L^\infty_{[t_a,t_b]}L^2_x} + \|w\|_{L^4_{[t_a,t_b]}L^\infty_x}\right) \|u_{\trans}\|^2_{L^\infty_{[t_a,t_b]}L^\infty_x}\\
&\le c (t_b - t_a)^{3/4}\|w\|_X.
\end{align*}

Seventh we treat terms of the form $wu_{\trans}u_{\bd}$ and $wu_{\refl}u_{\bd}$.
\begin{align*}
\indentalign \bigg\| \int_{t_a}^t e^{-i(t-s)H_q} w(s) u_{\trans}(s)u_{\bd}(s) ds \bigg\|_{L^p_{[t_a,t_b]}L^r_x} \\
&\le c \left(\|w u_{\trans}u_{\bd}\|_{L^1_{[t_a,t_b]}L^2_x} + \|Pwu_{\trans}u_{\bd}\|_{L^1_{[t_a,t_b]}L^r_x} \right) \\
&\le c (1 + |q|^{\frac 1 2})\|w\|_{L^1_{[t_a,t_b]}L^\infty_x} \|u_{\trans}\|_{L^\infty_{[t_a,t_b]}L^\infty_x} \|u_{\bd}\|_{L^\infty_{[t_a,t_b]}L^2_x}\\
&\le c (t_b-t_a)^\frac 34 |q|^{\frac 12} \|u_{\bd}\|_{L^\infty_{[t_a,t_b]}L^2_x} \|w\|_X.\\
&\le c (t_b-t_a)^\frac 34 \|w\|_X.
\end{align*}
where in the last step we again use $|q|^{\frac 12} \|u_{\bd}\|_{L^\infty_{[t_a,t_b]}L^2_x} \le 1$.

Eighth we treat terms of the form $wu_{\bd}^2$.
\begin{align*}
\indentalign \left\| \int_{t_a}^t e^{-i(t-s)H_q} w(s) u_{\bd}^2(s) ds \right\|_{L^p_{[t_a,t_b]}L^r_x} \\
&\le c \left(\|w u_{\bd}^2\|_{L^{4/3}_{[t_a,t_b]}L^1_x} + \|Pwu_{\bd}^2\|_{L^1_{[t_a,t_b]}L^r_x} \right) \\
&\le c |q| \|wu_{\bd}^2\|_{L^{4/3}_{[t_a,t_b]}L^1_x} \\
&\le c |q| \|w\|_{L^{4/3}_{[t_a,t_b]}L^\infty_x}\|u_{\bd}^2\|_{L^\infty_{[t_a,t_b]}L^1_x} \\
&\le c (t_b-t_a)^\frac 12 \|w\|_X
\end{align*}

Ninth we treat terms of the form $u_{\trans}^2u_{\refl}$ and $u_{\trans}u_{\refl}^2$, once again writing out the computation in one case only.
\begin{align*}
\indentalign \left\| \int_{t_a}^t e^{-i(t-s)H_q} u_{\trans}(s)^2u_{\refl}(s) ds \right\|_{L^p_{[t_a,t_b]}L^r_x} \\
&\le c \left(\|u_{\trans}^2u_{\refl}\|_{L^{4/3}_{[t_a,t_b]}L^1_x} + \|Pu_{\trans}^2u_{\refl}\|_{L^1_{[t_a,t_b]}L^r_x} \right) \\
&\le c |q| \|u_{\trans}^2u_{\refl}\|_{L^\infty_{[t_a,t_b]}L^1_x} \\
&\le c |q| \|u_{\trans}u_{\refl}\|_{L^\infty_{[t_a,t_b]}L^2_x} \|u_{\trans}\|_{L^\infty_{[t_a,t_b]}L^2_x}.
\intertext{At this point we use the explicit formula for $u_{\trans}$ to bound $\|u_{\trans}\|_{L^\infty_{[t_a,t_b]}L^2_x}$, and we split up $\|u_{\trans}u_{\refl}\|_{L^\infty_{[t_a,t_b]}L^2_x}$ into regions along the positive and negative real axis.}
&\le c |q| \bigg(\|u_{\trans}\|_{L^\infty_{[t_a,t_b]}L^\infty_x}\|u_{\refl}\|_{L^\infty_{[t_a,t_b]}L^2_{x>0}} \\
&\hspace{2cm}+ \|u_{\trans}\|_{L^\infty_{[t_a,t_b]}L^2_{x<0}}\|u_{\refl}\|_{L^\infty_{[t_a,t_b]}L^\infty_x}\bigg) \\
&\le c(k) |q| \frac {(\log v)^{\sigma(k)}}{v^{k\eps}}.
\end{align*}

Tenth we treat terms of the form $u_{\bd}u_{\trans}^2$, $u_{\bd}u_{\trans}u_{\refl}$ and $u_{\bd}u_{\refl}^2$.
\begin{align*}
\indentalign \left\| \int_{t_a}^t e^{-i(t-s)H_q} u_{\bd}(s)u_{\trans}(s)^2 ds \right\|_{L^p_{[t_a,t_b]}L^r_x} \\
&\le c \left(\|u_{\bd}u_{\trans}^2\|_{L^1_{[t_a,t_b]}L^2_x} + \|Pu_{\bd}u_{\trans}^2\|_{L^1_{[t_a,t_b]}L^r_x} \right) \\
&\le c |q|^{\frac 1 2} \|u_{\bd}u_{\trans}^2\|_{L^\infty_{[t_a,t_b]}L^2_x} \\
&\le c |q|^\frac 12\|u_{\bd}u_{\trans}\|_{L^\infty_{[t_a,t_b]}L^2_x} \\
&\le c \left(\|u_\bd(t_a)\|_{L^2_x} + |q| \la q \ra^\frac 12 \|u_{\bd}(t_a)\|^3_{L^2_x} \right).
\end{align*}

Eleventh, and last, we treat terms of the form $u_{\bd}^2 u_{\trans}$ and $u_{\bd}^2 u_{\refl}$.
\begin{align*}
\indentalign \left\| \int_{t_a}^t e^{-i(t-s)H_q} u_{\bd}(s)^2u_{\trans}(s) ds \right\|_{L^p_{[t_a,t_b]}L^r_x} \\
&\le c \left(\|u_{\bd}^2u_{\trans}\|_{L^{4/3}_{[t_a,t_b]}L^1_x} + \|Pu_{\bd}^2u_{\trans}\|_{L^1_{[t_a,t_b]}L^r_x} \right) \\
&\le c |q| \|u_{\bd}^2u_{\trans}\|_{L^\infty_{[t_a,t_b]}L^1_x} \\
&\le c |q|\|u_\bd\|_{L^\infty_{[t_a,t_b]}L^2_x} \|u_{\bd}u_{\trans}\|_{L^\infty_{[t_a,t_b]}L^2_x}.
\end{align*}
But using once again $|q|^\frac 12\|u_\bd\|_{L^\infty_{[t_a,t_b]}L^2_x} \le 1$ we see that this obeys the same bound as the previous term.

Combining these eleven estimates, we find that
\begin{align*}
\|w\|_X \le c \Big(&\|w(t_a)\|_{L^2_x} + c(k)\la q\ra^2\frac{(\log v)^{\sigma(k)}}{v^{k\eps}} + \|u_\bd(t_a)\|_{L^2_x} + |q| \la q \ra^\frac 12 \|u_{\bd}(t_a)\|^3_{L^2_x}\\
&+ (t_b-t_a)^\frac 12\|w\|_X + \|w\|_X^2 + \|w\|_X^3 \Big).
\end{align*}
Choosing $c_1$ sufficiently small allows us to absorb the linear term into the left hand side. A continuity argument just like that in Phase 1 now gives us the conclusion.
\end{proof}
 
\section{Phase 2 for $q \ge 0$}
\label{S:phase2pos}

In this section, we prove Theorem \ref{th:3}.  We repeat the three lemmas of Phase 2 in the case $q \ge 0$. These results are used in the course of Phase 2 in the case $q<0$ in order to approximate the flow of $e^{-itH_0}$ by that of $\nlso$, and they also give an improvement of the asymptotic result in \cite{HMZ} for $q >0$.

\begin{lem}
\label{linappp}
Let $\phi\in L^2$ and $0<t_b$.  If $q \ge 0$ and $t_b^{1/2} \le c_1 \|\phi\|_{L^2}^{-2}$, then
\begin{equation}
\|\nlsq(t)\phi - e^{-itH_q}\phi\|_{L^P_{[0,t_b]}L^R_x} \le c_2 t_b^{1/2}\|\phi\|_{L^2}^3,
\end{equation}
where $P$ and $R$ satisfy $\frac 2 P + \frac 1 R = \frac 1 2$, and $c_1$ and $c_2$ depend only on the constant appearing in the Strichartz estimates.
\end{lem}

\begin{proof}
Let $h(t) = \nlsq(t)\phi$, so that
\[
i\D_t h + \frac 1 2 \D_x^2 h - q \delta_0(x)h + |h|^2h = 0.
\]
We use again the notation $\|h\|_{X'} = \|h\|_{L^\infty_{[0,t_b]}L^2_x} + \|h\|_{L^6_{[0,t_b]}L^6_x}$. We apply the Strichartz estimate, which in this case reads
\[
\|h\|_{L^p_{[0,t_b]}L^r_x} \le c(\| \phi \|_{L^2} + \||h|^2h\|_{L^{\tilde p}_{[0,t_b]}L^{\tilde r}_x })
\]
once with $(p,r) = (\infty, 2)$ and $(\tilde p, \tilde r) = (6/5,6/5)$, and once with $(p,r) = (6,6)$ and $(\tilde p, \tilde r) = (6/5,6/5)$.
As before, the $h$ terms are each estimated using H\"older's inequality as
\begin{align*}
\|h\|^3_{L^{18/5}_{[0,t_b]}L^{18/5}_x} &\le c t_b^{1/2} \|h\|^3_{X'},
\end{align*}
yielding
\[
\|h\|_{X'} \le c(\|\phi\|_{L^2} + t_b^{1/2} \|h\|^3_{X'}).
\]
Using the continuity of $\|h\|_{X'(t_b)}$, we conclude that
\[
\|h\|_{X'} \le 2c\|\phi\|_{L^2},
\]
so long as $8c^2t_b^{1/2}\|\phi\|_{L^2}^2 \le 1$.

We now apply the Strichartz estimate to $u(t) = h(t) - e^{-itH_q}\phi$, observing that the initial condition is zero and the effective forcing term $-|h|^2h$, to get
\[ \|h(t) - e^{-itH_q}\phi\|_{L^P_{[0,t_b]}L^R_x} \le c \||h|^2h\|_{L^{6/5}_{[0,t_b]}L^{6/5}_x} \le ct_b^{1/2}\|h\|_{X'}^3 \le ct_b^{1/2}\|\phi\|_{L^2}^3.\]
\end{proof}

\begin{lem}
\label{L:approx2p}
Under the same hypotheses as the previous lemma,
\begin{equation}
\label{E: approx2p}
\|\nlsq(t)\phi - g\|_{L^\infty_{[0,t_b]}L^2_x} \le  c \left(t_b^2\|\phi\|_{L^2}^9 + t_b^{3/2}\|\phi\|_{L^2}^7 + t_b \|\phi\|_{L^2}^5\right),
\end{equation}
where
\[
g(t) = e^{-itH_q}\phi + \int_0^t e^{-i(t-s)H_q}|e^{-isH_q}\phi|^2e^{-isH_q}\phi ds.
\]
\end{lem}

\begin{proof}
A direct calculation shows
\[
h(t) - g(t)  = \int_0^t e^{-i(t-s)H_q}\left(|h(s)|^2h(s) - |e^{-isH_q}\phi|^2e^{-isH_q}\phi\right)ds.
\]
The Strichartz estimate gives us in this case
\[ \|h - g\|_{L^\infty_{[0,t_b]}L^2_x} \le \||h|^2h - |e^{-isH_q}\phi|^2e^{-isH_q}\phi\|_{L^{6/5}_{[0,t_b]}L^{6/5}_x}. \]

We introduce the notation $w(t) = h(t) - e^{-itH_q}\phi$, and use this to rewrite our difference of cubes:
\[
|h|^2h - |e^{-isH_q}\phi|^2e^{-isH_q}\phi = 
\begin{aligned}[t]
&w|w|^2 + 2 e^{-isH_q}\phi |w|^2 + e^{isH_q}\overline{\phi} w^2 \\
&+ 2|e^{-isH_q}\phi|^2w + \left(e^{-isH_q}\phi\right)^2\ow.
\end{aligned}
\]
We proceed term by term, using H\"older estimates similar to the ones in the previous lemma and in Phase 1. Our goal is to obtain Strichartz norms of $w$, so that we can apply Lemma \ref{linappp}.

We have, for the cubic term,
\begin{align*}
\|w^3\|_{L^{6/5}_{[0,t_b]}L^{6/5}_x} &\le c t_b^{1/2}\|w\|_{L^\infty_{[t_a,t_b]}L^2_x}\|w\|^2_{L^6_{[t_a,t_b]}L^6_x} \\
&\le c t_b^2\|\phi\|_{L^2}^9.
\end{align*}
For the first inequality we used H\"older, and for the second Lemma \ref{linapp}. Next we treat the quadratic terms using the same strategy (observe that as before we ignore complex conjugates):
\begin{align*}
\|e^{-isH_q}\phi |w|^2\|_{L^{6/5}_{[0,t_b]}L^{6/5}_x} &\le c t_b^{1/2}\|w\|_{L^\infty_{[t_a,t_b]}L^2_x}\|w\|_{L^6_{[t_a,t_b]}L^6_x}\|e^{-isH_q}\phi\|_{L^6_{[t_a,t_b]}L^6_x} \\
&\le c t_b^{3/2}\|\phi\|_{L^2}^7.
\end{align*}
And finally the linear terms:
$$\||e^{-isH_q}\phi|^2 w\|_{L^{6/5}_{[0,t_b]}L^{6/5}_x} \le c t_b^{1/2}\|w\|_{L^\infty_{[t_a,t_b]}L^2_x}\|e^{-isH_q}\phi\|^2_{L^6_{[t_a,t_b]}L^6_x} \le c t_b \|\phi\|_{L^2}^5.$$

Putting all this together, we see that
\[ \|h - g\|_{L^\infty_{[0,t_b]}L^2_x} \le c \left(t_b^2\|\phi\|_{L^2}^9 + t_b^{3/2}\|\phi\|_{L^2}^7 + t_b \|\phi\|_{L^2}^5\right).\]
\end{proof}

Finally we estimate, at time $t_2$, the error incurred by dropping the integral term:

\begin{lem}
\label{L:dropintp}
For $t_1 < t_2$ and $\phi = u(x,t_1)$, we have
\begin{align}\label{E:dropintp}
\bigg\| \int_{t_1}^{t_2} e^{-i(t_2-s)H_q}&|e^{-isH_q}\phi|^2e^{-isH_q}\phi ds \bigg\|_{L^2_x} \le \\
\nonumber&c \left[(t_2-t_1) + (t_2-t_1)^{\frac 12}\left(q e^{-v^{1-\delta}} + q^2e^{-2v^{1-\delta}} + q^3 e^{-3v^{1-\delta}}\right)\right],
\end{align}
where $c$ is independent of the parameters of the problem.
\end{lem}

\begin{proof}
We write $\phi(x) = \phi_1(x) + \phi_2(x)$, where $\phi_1(x) = e^{-it_1v^2/2}e^{it_1/2}e^{ixv}\sech(x-x_0-vt_1)$, and estimate individually the eight resulting terms. We know that for large $v$, $\phi_2$ is exponentially small in $L^2$ norm from Lemma 3.1 of \cite{HMZ}. This makes the term which is cubic in $\phi_1$ the largest, and we treat this one first.

I. We claim $\left\| \int_{t_1}^{t_2} e^{-i(t_2-s)H_q}|e^{-isH_q}\phi_1|^2e^{-isH_q}\phi_1 ds \right\|_{L^2_x} \le c (t_2-t_1)$.

We begin with a direct computation
\begin{align*}
\indentalign \left\| \int_{t_1}^{t_2} e^{-i(t_2-s)H_q}|e^{-isH_q}\phi_1|^2e^{-isH_q}\phi_1 ds \right\|_{L^2_x} \\
&\le (t_2 - t_1) \left\| e^{-i(t_2-s)H_q}|e^{-isH_q}\phi_1|^2e^{-isH_q}\phi_1 \right\|_{L^\infty_{[t_1,t_2]}L^2_x} \\
&\le c (t_2-t_1) \left\| |e^{-isH_q}\phi_1|^2e^{-isH_q}\phi_1 \right\|_{L^\infty_{[t_1,t_2]}L^2_x}.
\end{align*}
It remains to show that this last norm is bounded by a constant. This follows exactly the same argument as that given in part I of Lemma \ref{L:dropint}, with the difference that terms involving $P$ are omitted.

II. For the other terms we use simpler Strichartz estimates. The smallness will come more from the smallness of $\phi_2$ than from the brevity of the time interval.
\begin{align*}
\indentalign \bigg\| \int_{t_1}^{t_2} e^{-i(t_2-s)H_q}|e^{-isH_q}\phi_1|^2e^{-isH_q}\phi_2 ds \bigg\|_{L^2_x} \\
&\le c \||e^{-isH_q}\phi_1|^2e^{-isH_q}\phi_2\|_{L^1_{[t_1,t_2]}L^2_x}.
\intertext{We have used the Strichartz estimate with $(\tilde p, \tilde r) = (1,2)$, and use H\"older's inequality so as to put ourselves in a position to reapply the Strichartz estimate.}
&\le c (t_2-t_1)^{\frac 12}\|e^{-isH_q}\phi_1\|^2_{L^4_{[t_1,t_2]}L^\infty_x}\|e^{-isH_q}\phi_2\|_{L^\infty_{[t_1,t_2]}L^2_x} \\
&\le c (t_2-t_1)^{\frac 12} \|\phi_1\|_{L^2_x}^2\|\phi_2\|_{L^2_x} \\
&\le c (t_2-t_1)^{\frac 12}q e^{-v^{1-\delta}}.
\end{align*}
Note that here $\delta = 1 -\eps$. Similarly we find that
\begin{align*}
\bigg\| \int_{t_1}^{t_2} e^{-i(t_2-s)H_q}|e^{-isH_q}\phi_2|^2e^{-isH_q}\phi_1 ds \bigg\|_{L^2_x} &\le c (t_2-t_1)^{\frac 12}q^2 e^{-2v^{1-\delta}} \\
\bigg\| \int_{t_1}^{t_2} e^{-i(t_2-s)H_q}|e^{-isH_q}\phi_2|^2e^{-isH_q}\phi_2 ds \bigg\|_{L^2_x} &\le c (t_2-t_1)^{\frac 12}q^3 e^{-3v^{1-\delta}}.
\end{align*}
\end{proof}

These three lemmas improve the error in the case $q>0$ from $v^{-\frac {1-\eps} 2}$ to $v^{-(1-\eps)}$, or, in the language of \cite{HMZ}, from $v^{-\frac \delta 2}$ to $v^{-\delta}$.

\end{document}